\documentclass[12pt,a4paper]{amsart}
\usepackage[czech,english]{babel}
\usepackage{amssymb,amsxtra}
\usepackage{mathrsfs}
\usepackage{graphicx}
\usepackage[all,cmtip]{xy}
\usepackage{hyperref}

\calclayout

\newcommand{\+}{\protect\nobreakdash-}

\renewcommand{\:}{\colon}

\newcommand{\rarrow}{\longrightarrow}
\newcommand{\larrow}{\longleftarrow}

\newcommand{\tim}{\rightthreetimes}

\newcommand\mapsfrom{\mathrel{\reflectbox{\ensuremath{\longmapsto}}}}

\newcommand{\lrarrow}{\mskip.5\thinmuskip\relbar\joinrel\relbar\joinrel
 \rightarrow\mskip.5\thinmuskip\relax}

\DeclareMathOperator{\Hom}{Hom}

\DeclareMathOperator{\End}{End}

\DeclareMathOperator{\id}{id}

\DeclareMathOperator{\im}{im}
\DeclareMathOperator{\rad}{rad}

\newcommand{\Modl}{{\operatorname{\mathsf{--Mod}}}}
\newcommand{\Modr}{{\operatorname{\mathsf{Mod--}}}} 
\newcommand{\Modcs}{{\operatorname{\mathsf{Mod}_\cs\mathsf{--}}}} 

\newcommand{\Contra}{{\operatorname{\mathsf{--Contra}}}}
\newcommand{\Discr}{{\operatorname{\mathsf{Discr--}}}}

\newcommand{\Add}{\mathsf{Add}}
\newcommand{\Prod}{\mathsf{Prod}}

\newcommand{\proj}{\mathsf{proj}}

\newcommand{\cs}{\mathsf{cs}}

\newcommand{\rop}{{\mathrm{op}}}
\newcommand{\sop}{{\mathsf{op}}}

\newcommand{\cont}{{\mathrm{cont}}}

\newcommand{\Mat}{\mathfrak{Mat}}

\newcommand{\A}{\mathfrak A}

\newcommand{\C}{\mathfrak C}
\newcommand{\D}{\mathfrak D}

\newcommand{\HH}{\mathfrak H}
\newcommand{\I}{\mathfrak I}
\newcommand{\J}{\mathfrak J}

\renewcommand{\P}{\mathfrak P}
\newcommand{\Q}{\mathfrak Q}
\newcommand{\R}{\mathfrak R}
\renewcommand{\S}{\mathfrak S}

\newcommand{\U}{\mathfrak U}

\newcommand{\sA}{\mathsf A}

\newcommand{\scrM}{\mathscr M}
\newcommand{\scrN}{\mathscr N}

\newcommand{\boZ}{\mathbb Z}

\newcommand{\Section}[1]{\bigskip\section{#1}\medskip}
\newcommand{\SectionSpecialChar}[2]{\bigskip\section{\texorpdfstring{#1}{#2}}\medskip}
\theoremstyle{plain}
\newtheorem{thm}{Theorem}[section]

\newtheorem{lem}[thm]{Lemma}
\newtheorem{prop}[thm]{Proposition}
\newtheorem{cor}[thm]{Corollary}
\theoremstyle{definition}
\newtheorem{ex}[thm]{Example}

\newtheorem{rem}[thm]{Remark}
\newtheorem{rems}[thm]{Remarks}

\begin{document}

\title{Topologically semiperfect topological rings}

\author{Leonid Positselski}

\address{Leonid Positselski, Institute of Mathematics
of the Czech Academy of Sciences \\
\v Zitn\'a~25, 115~67 Prague~1 \\ Czech Republic} 

\email{positselski@math.cas.cz}

\author{Jan \v S\v tov\'\i\v cek}

\address{Jan {\v S}{\v{t}}ov{\'{\i}}{\v{c}}ek, Charles University,
Faculty of Mathematics and Physics, Department of Algebra,
Sokolovsk\'a 83, 186 75 Praha, Czech Republic}

\email{stovicek@karlin.mff.cuni.cz}

\begin{abstract}
 We define topologically semiperfect (complete, separated, right linear)
topological rings and characterize them by equivalent conditions.
 We show that the endomorphism ring of a module, endowed with
the finite topology, is topologically semiperfect if and only if
the module is decomposable as an (infinite) direct sum of modules
with local endomorphism rings.
 Then we study structural properties of topologically semiperfect
topological rings and prove that their topological Jacobson radicals
are strongly closed and the related topological quotient rings are
topologically semisimple.
 For the endomorphism ring of a direct sum of modules with local
endomorphism rings, the topological Jacobson radical is described
explicitly as the set of all matrices of nonisomorphisms.
 Furthermore, we prove that, over a topologically semiperfect topological
ring, all finitely generated discrete modules have projective covers in
the category of modules, while all lattice-finite contramodules have
projective covers in both the categories of modules and contramodules.
 We also show that the topological Jacobson radical of a 
topologically semiperfect topological ring is equal to the closure of
the abstract Jacobson radical, and present a counterexample demonstrating
that the topological Jacobson radical can be strictly larger than
the abstract one.
 Finally, we discuss the problem of lifting idempotents modulo
the topological Jacobson radical and the structure of projective contramodules
for topologically semiperfect topological rings.
\end{abstract}

\maketitle

\tableofcontents

\section{Introduction}
\medskip

 Topological algebra generally, and topological ring theory in particular,
is an old discipline, going back to the applications of the metric or
topological fields of real and $p$\+adic numbers in algebraic number
theory, Jacobson's density theorem, etc.
 In recent years, it received new impetus from the discovery of
the abelian categories of contramodules over topological rings
(see the long overview paper~\cite{Prev}).

 From our point of view, the natural generality is achieved in
the setting of complete, separated, \emph{right linear} topological
rings~$\R$.
 The latter condition means that open right ideals form a base of
neighborhoods of zero in~$\R$.
 To such a topological ring one assigns two abelian categories:
the Grothendieck abelian category of \emph{discrete right\/
$\R$\+modules}, and the locally presentable abelian category of 
\emph{left\/ $\R$\+contramodules} (which has enough projective objects).
 Having such abelian module categories at one's disposal, one can
start on the program of extending the classical concepts of ring theory
to the topological realm.

 The first step, viz., an infinite-dimensional version of
the Wedderburn--Artin theorem in the context of topological rings,
was suggested in the paper~\cite{IMR}.
 The authors of~\cite{IMR} preferred pseudo-compact modules to discrete
ones, and they did not consider contramodules; so they only had modules
on the one side over their topological rings.
 The topological Wedderburn--Artin theorem of~\cite[Theorem~3.10]{IMR}
was augmented in~\cite[Section~6]{PS3}, where contramodules were
thrown into the game.
 Thus appeared the concept of a \emph{topologically semisimple}
topological ring.

 The next step was made in the present authors' paper~\cite{PS3}
(with a preceding attempt in~\cite{Pproperf} and a further development
in~\cite{BPS}), where we defined the notion of a \emph{topologically
left perfect} (right linear) topological ring.
 Some of the properties characterizing perfect rings in
the classical~\cite[Theorem~P]{Bas} were proved to remain equivalent
in the topological realm in~\cite[Section~14]{PS3}, while another
characterization of perfect rings from~\cite[Theorem~P]{Bas} was shown
to be equivalent to the rest of the bunch in the topological context if
and only if a certain open problem~\cite[Question~1 in Section~2]{AS}
has positive answer.

 The topological algebra of right linear topological rings is
intertwined with module theory, particularly with the theory of
\emph{direct sum decompositions} of modules.
 The endomorphism ring of any module over an associative ring is
naturally endowed with the \emph{finite topology}, making it a complete,
separated right linear topological ring; and conversely, any complete,
separated right linear topological ring arises as the endomorphism ring
of a module~\cite[Section~4]{PS3}.
 Fundamentally, the connection between decomposition theory of a module
over a ring and contramodule theory over its endomorphism ring is
provided by the result of~\cite[Theorem~7.1]{PS1}.

 The topologically semisimple topological rings are the endomorphism
rings of (infinitely generated) semisimple modules~\cite[Section~6]{PS3}.
 The topologically perfect topological rings are the endomorphism rings
of \emph{modules with perfect decomposition}~\cite[Section~10]{PS3}.
 The interaction is mutually beneficial, providing applications in both
directions: the proof of the characterization of topologically
perfect topological rings in~\cite[Theorem~14.1]{PS3} is based on
known results in the theory of direct sum decompositions of
modules~\cite[Theorem~1.4]{AS}, and conversely, the positive answer
to~\cite[Question~1 in Section~2]{AS} for countably generated modules
obtained in~\cite[Theorem~12.2]{PS3} uses topological ring theory.

\smallskip

 In this paper, we make a further step on the road of generalization
and define \emph{topologically semiperfect} topological rings.
 The notion of a topologically semiperfect topological ring is
the topological algebra counterpart of the classical concept of
a module decomposable into a direct sum of modules \emph{with local
endomorphism rings} (such modules form the natural generality for
the Krull--Schmidt--Remak--Azumaya theorem about uniqueness of
direct sum decompositions~\cite[Theorem~12.6]{AF},
\cite[Theorem~2.12]{Fac}).
 Specifically, a module over a ring is decomposable as a (possibly
infinite) direct sum of modules with local endomorphism rings if
and only if its endomorphism ring, endowed with the finite topology,
is topologically semiperfect.
 This is a particular case of our
Proposition~\ref{semiperfect-decomposition}.

 The classical ring-theoretic notion of a semiperfect ring, unlike
that of a perfect ring, is left-right symmetric.
 There are several equivalent definitions of semiperfect rings,
including those in terms of the ring's structural properties
(the quotient ring by the Jacobson radical is semisimple and idempotents
lift modulo the Jacobson radical), those with representation-theoretic
flavor (the regular module is a direct sum of modules with local
endomorphism rings), and in terms of existence of projective covers
for all simple modules, or equivalently, for all finitely generated
modules.

 As it was the case for the theory of topological perfectness developed
in~\cite{PS3}, in the context of semiperfectness we have been likewise
unable to reproduce, in our topological setting, the full picture of 
equivalent characterizations known for discrete rings in general.
 But we have managed to prove some equivalences and some implications
and obtained an almost full picture for topological rings with
countable bases of neighborhoods of zero
(see Theorems~\ref{top-semiperfectnes-via-proj-covers-thm}
and~\ref{projective-contramodules-for-countable-basis-thm}).

 First of all, in the topological setting the left-right symmetry is
lost from the outset: the topological rings we are interested in are
\emph{right linear} or, if one prefers, one can switch the sides and
make them left linear, but not both (the class of two-sided linear
topological rings is too narrow to include our intended examples of
the topological endomorphism rings).
 We define topologically semiperfect topological rings by the direct
sum decomposition property of their left regular contramodule, or
equivalently, direct product decomposition property of the right
regular topological module, and show that this is equivalent to
existence of an infinite complete zero-convergent family of orthogonal
local idempotents (Theorem~\ref{topologically-semiperfect-rings}).

 Then we prove that the topological semiperfectness implies
topological semisimplicity of the topological quotient ring by
the topological Jacobson radical
(Theorem~\ref{topologically-semiperfect-structural-properties})
and existence of projective covers for finitely generated discrete right modules and lattice-finite left contramodules,
i.e.\ contramodules which are finite in their lattice of subcontramodules
(Propositions~\ref{discrete-fg-projective-covers}
and~\ref{lattice-finite-contramodules-projective-covers}).
When the topological ring has a countable base of neighborhoods of zero,
the existence of projective covers of simple, or equivalently finitely
generated, discrete right modules actually characterizes topological
semiperfectness (Theorem~\ref{top-semiperfectnes-via-proj-covers-thm}).
 In Section~\ref{lifting-idempotents-secn}, we
discuss the problem of lifting idempotents modulo the topological
Jacobson radical of a topologically semiperfect topological ring.
 We show that finite families of orthogonal primitive idempotents can be
lifted (Proposition~\ref{lifting-orthogonal-primitive-idempotents}),
and present an example illustrating the difficulties one runs into
when one attempts to lift infinite zero-convergent families of
orthogonal idempotents.
Finally, in the last Section~\ref{structure-of-projective-contramodules-secn} we show that over a topologically semiperfect topological ring with countable base of neighborhoods of zero, all projective contramodules decompose as coproducts of projective covers of simple contramodules.
This can be viewed as a topological algebraic manifestation of the Crawley--J\o{}nsson--Warfield structure theorem on direct summands of direct sums of countably generated modules with local endomorphism rings \cite[Theorem 26.5]{AF}.
We do not known whether the countability assumption on the base of neighborhoods of zero can be dropped here,
which corresponds to the open problem of whether the countable generatedness can be dropped in the Crawley--J\o{}nsson--Warfield theorem.

The anonymous referee kindly made us aware of the paper~\cite{Gre} by Enrico Gregorio, who also studied right linear (and often complete and separated) topological rings and defined a notion of topological semiperfectness. In fact, he introduced it in two different flavors: topologically t-semiperfect topological rings and topologically d-semiperfect topological rings.
However, except for giving fundamentals about the topological Jacobson radical in~\cite[Section~1]{Gre} and for \cite[Proposition 2.12 and Theorem 2.14]{Gre} serving as inspiration for our Section~\ref{surjective-homomorphisms-secn}, the present paper and \cite{Gre} have relatively little in common.
The notion of t-semiperfectness from~\cite{Gre} is (at least apparently, we are presently not aware of any example) weaker than our topological semiperfectness and it seems to be difficult to build a theory on that definition,
while d-semiperfectness is on the other hand very restrictive.

\smallskip

 Some words are in order about the topological Jacobson radical.
 There are two notions of a Jacobson radical for a right linear
topological ring~$\R$: the classical abstract Jacobson radical~$H$
(which ignores the topology) and the topological Jacobson radical~$\HH$.
 The topological Jacobson radical is defined as the intersection of
all the \emph{open} maximal right ideals in~$\R$ (so $H$ is obviously
a subset of~$\HH$).
 The topological Jacobson radical was first defined by Gregorio in~\cite[Section~1]{Gre}.
 Much later the concept was rediscovered in~\cite{IMR}, and it was
shown in~\cite[Theorem~3.8(3)]{IMR} that for a \emph{two-sided linear}
topological ring the two Jacobson radicals coincide (cf.\
the discussion in~\cite[Section~7]{Pproperf}).
 For a topologically left perfect (right linear) topological ring,
the two Jacobson radicals also coincide~\cite[Lemma~10.3]{PS3}.

 In this paper, we show that the topological Jacobson radical of
a topologically semiperfect (right linear) topological ring is
equal to the topological closure of the abstract Jacobson radical
(Corollary~\ref{Jacobson-closure}), but the abstract Jacobson
radical need not be closed in the topology, and accordingly,
the topological Jacobson radical can be strictly larger than
the abstract one
(Example~\ref{topological-jacobson-orthogonalization-counterex}(1)).
 This phenomenon is related to the difficulties we encounter
with the problem of lifting of infinite families of orthogonal
idempotents, as we explain in
Example~\ref{topological-jacobson-orthogonalization-counterex}(2).

\subsection*{Acknowledgment}
 We would like to thank the anonymous referee for careful reading of
the manuscript, and particularly for suggesting the most relevant
reference~\cite{Gre}.
 This research is supported by GA\v CR project 20-13778S.
 The first-named author is also supported by research plan
RVO:~67985840.

\Section{Preliminaries} \label{prelim-secn}

\subsection{Right linear topological rings} \label{topological-rings-subsecn}

 All \emph{rings} in this paper are presumed to be associative and
unital, and all modules are presumed to be unital.
 All the \emph{ring homomorphisms} take the unit to the unit.
 As usual, the abelian categories of right and left modules over
a ring $R$ are denoted by $\Modr R$ and $R\Modl$, respectively.

 A topological abelian group $A$ is said to have \emph{linear topology}
if open subgroups form a base of neighborhoods of zero in~$A$.
 We recall that the \emph{completion} $\A$ of a topological abelian
group $A$ with linear topology is defined as the projective limit
$\A=\varprojlim_{U\subset A}A/U$, where $U$ ranges over all the open
subgroups of $A$ (or equivalently, all the open subgroups belonging to
a chosen base of neighborhoods of zero consisting of open subgroups).
 The topological abelian group $A$ is said to be \emph{separated} if
the natural completion map $\lambda_A\:A\rarrow\A$ is injective, and
$A$ is said to be \emph{complete} if the map $\lambda_A$ is surjective.
 We refer to the paper~\cite{Pextop} for an extensive background on
topological abelian groups with linear topology.

 In the rest of this paper, all topological abelian groups are
presumed to have linear topology, and to be complete and separated.
 A topological ring is said to be \emph{right linear} if it has a base
of neighborhoods of zero consisting of open right ideals.
 All topological rings in this paper are presumed to be complete,
separated, and right linear.
 We refer to~\cite[Section~2]{Pproperf} for the background material
on topological rings.

\subsection{Discrete and complete separated modules} \label{discrete-and-cs-modules-subsecn}

 Let $\R$ be a topological ring.
 A right $\R$\+module $M$ is said to be \emph{discrete} if
the annihilator of any element of $M$ is an open right ideal in~$\R$;
equivalently, this means that the action map $M\times\R\rarrow M$ is
continuous in the discrete topology of $M$ and the given topology
on~$\R$.
 The full subcategory of discrete right $\R$\+modules $\Discr\R\subset
\Modr\R$ is a hereditary pretorsion class in $\Modr\R$ and
a Grothendieck abelian category~\cite[Section~2.4]{Pproperf}.

We will also consider a bigger category $\Modcs\R\supset\Discr\R$ of complete separated topological $\R$-modules with continuous $\R$-module homomorphisms. The objects are modules $\scrM\in\Modr\R$ equipped with a linear topology such that open $\R$-submodules form a base of neighborhoods of zero in~$\scrM$ and that the action map $\scrM\times\R\rarrow\scrM$ is continuous. This category is typically not abelian, but it is additive, idempotent-complete and has products.

Indeed, it is easy to show that a difference between continuous homomorphisms of topological right $\R$\+modules is again continuous.
It is standard that, given an idempotent continuous homomorphism $e\:\scrM\rarrow\scrM$, one has $\scrM \simeq e\scrM\oplus(1-e)\scrM$ in $\Modcs\R$, where $e\scrM, (1-e)\scrM$ carry the subspace topologies, which coincide with the quotient topologies (see~\cite[Example 3.6(2)]{PS3}). Finally, a product $\prod_{x\in X}\scrM_x$ of a family of topological modules from $\Modcs\R$, equipped with the product topology, is a product in $\Modcs\R$.

Given $\scrM,\scrN\in\Modcs\R$, we will denote the group of continuous $\R$\+module homomorphisms by $\Hom_\R^\cont(\scrM,\scrN)$.
Note that $\R$ is naturally an object of $\Modcs\R$, as are all summands $e\R$, where $e=e^2\in\R$ is an idempotent element (in this case, the left multiplication $e\cdot-$ on $\R$ is always continuous).

\subsection{Contramodules} \label{contramodules-subsecn}

 Given an abelian group $A$ and a set $X$, we use $A[X]=A^{(X)}$ as
a notation for the direct sum of $X$ copies of~$A$.
 The elements of $A[X]$ are interpreted as finite formal linear
combinations of elements of $X$ with the coefficients in~$A$.
 For a topological abelian group $\A$ and a set $X$, we put
$\A[[X]]=\varprojlim_{\U\subset\A}(\A/\U)[X]$, where the projective
limit is taken over all the open subgroups $\U\subset\A$ (or over
all the open subgroups belonging to a chosen base of neighborhoods
of zero consisting of open subgroups in~$\A$).
 The elements of $\A[[X]]$ are interpreted as infinite formal linear
combinations $\sum_{x\in X}a_xx$, where $(a_x\in\A)_{x\in X}$ is
a family of elements converging to zero in the topology of~$\A$.
 The latter condition means that, for every open subgroup $\U\subset\A$,
one has $a_x\in\U$ for all but a finite subset of indices $x\in X$
\cite[Section~2.5]{Pproperf}.

 For any (complete, separated) topological abelian group $\A$, the rule
assigning the set (or abelian group) $\A[[X]]$ to a set $X$ is
a covariant functor on the category of sets.
 Given a map of sets $f\:X\rarrow Y$, the induced map $\A[[f]]\:
\A[[X]]\rarrow\A[[Y]]$ takes a formal linear combination
$\sum_{x\in X}a_xx$ to the formal linear combination
$\sum_{y\in Y}b_yy$, where $b_y=\sum_{x\in X}^{f(x)=y}a_x$ for
every $y\in Y$.
 The infinite sum here is understood as the limit of finite partial
sums in the topology of~$\A$.

 For any (complete, separated, right linear) topological ring $\R$,
the functor $X\longmapsto\R[[X]]$ has a natural structure of
a \emph{monad} on the category of sets.
 This means that there are natural transformations of monad unit
$\epsilon_X\:X\rarrow\R[[X]]$ and monad multiplication
$\phi_X\:\R[[\R[[X]]]]\rarrow\R[[X]]$, defined for all sets $X$ and
satisfying the associativity and unitality axioms in the definition
of a monad.
 Here the monad unit $X\rarrow\R[[X]]$ is the ``point measure'' map
taking an element $x\in X$ to the formal linear combination
$\sum_{y\in Y}r_yy$ with $r_x=1$ and $r_y=0$ for $y\ne x$.
 The monad multiplication $\R[[\R[[X]]]]\rarrow\R[[X]]$ is
the ``opening of parentheses'' map assigning a formal linear combination
to a formal linear combination of formal linear combinations.
 Its construction uses the multiplication in $\R$ and the infinite
sums computed as the topological limits of finite partial
sums~\cite[Section~2.6]{Pproperf}, \cite[Sections~6.1--6.2]{PS1}.

 A \emph{left contramodule} over a topological ring $\R$ is
a module over the monad $X\longmapsto\R[[X]]$.
 Here we use the terminology \emph{modules over a monad} for what are
generally known as ``algebras over a monad''; the fact that contramodule
categories are additive explains this terminological preference.
 Specifically, this means that a left $\R$\+contramodule $\C$ is
a set endowed with a \emph{left contraaction map} $\pi_\C\:\R[[\C]]
\rarrow\C$ satisfying the associativity and unitality axioms of
an algebra/module over a monad.
 These axioms require that the two compositions 
$$
 \R[[\R[[\C]]]]\,\rightrightarrows\R[[\C]]\rarrow\C
$$
must be equal to each other, $\pi_\C\circ\R[[\pi_\C]]=
\pi_\C\circ\phi_\C$, while the composition
$$
 \C\rarrow\R[[\C]]\rarrow\C
$$
must be equal to the identity map, $\pi_\C\circ\epsilon_\C=\id_\C$. In particular, given any set $X$, a collection of elements $(c_x\in\C)_{x\in X}$ (which we can interpret as a map $f\:X\rarrow\C$) and a zero-convergent family $\sum_{x\in X}r_xx\in\R[[X]]$, we denote the image of $\sum_{x\in X}r_xx$ under the composition $\pi_\C\circ\R[[f]]$ simply by $\sum_{x\in X}r_xc_x\in\C$. In this sense, the contraaction informally just defines a good way to evaluate infinite formal $\R$\+linear combinations with zero-convergent families of coefficients in $\C$.

 In particular, in the case of a discrete ring $R$, when one has
$R[[X]]=R[X]$, the definition above becomes a fancy way to define
the usual left $R$\+modules~\cite[Section~6.1]{PS1}.
 For a topological ring $\R$ and a left $\R$\+contramodule $\C$,
one can compose the contraaction map $\R[[\C]]\rarrow\C$ with
the natural inclusion $\R[\C]\hookrightarrow\R[[\C]]$, producing a map
$\R[\C]\rarrow\C$ defining a left $\R$\+module structure on~$\C$.
 So any left $\R$\+contramodule has an underlying left $\R$\+module
structure, and we obtain a natural forgetful functor $\R\Contra
\rarrow\R\Modl$ from the category of left $\R$\+contramodules to
the category of left $\R$\+modules.
 The category of left $\R$\+contramodules is abelian with infinite
products and coproducts, and the forgetful functor $\R\Contra\rarrow
\R\Modl$ is exact.
 The latter functor also preserves infinite products, but \emph{not}
coproducts~\cite[Section~2.7]{Pproperf}, \cite[Section~6.2]{PS1}.
 Given a family of $\R$\+contramodules $\C_\alpha$, we denote by
$\coprod_\alpha\C_\alpha=\coprod_\alpha^{\R\Contra}\C_\alpha$
the coproduct of the objects $\C_\alpha$ in the category $\R\Contra$.

 For any left $\R$\+contramodules $\C$ and $\D$, the abelian group
of morphisms $\C\rarrow\D$ in $\R\Contra$ is denoted by
$\Hom^\R(\C,\D)$.
 For any set $X$, the set $\R[[X]]$ has a natural structure of left
$\R$\+contramodule with the contraaction map $\pi_{\R[[X]]}=\phi_X$.
 The $\R$\+contramodules of this form are called the \emph{free}
left $\R$\+contramodules.
 For any left $\R$\+contramodule $\C$, the abelian group of morphisms
$\Hom^\R(\R[[X]],\C)$ is isomorphic to the group of all maps of sets
$X\rarrow\C$; an arbitrary such map of sets can be uniquely extended
to a morphism from the free contramodule.
 There are enough projective objects in the abelian category
$\R\Contra$; an $\R$\+contramodule is projective if and only if it is
a direct summand of a free $\R$\+contramodule.

 For any ring $R$, left $R$\+module $M$, and additive subgroup
$A\subset R$ we denote by $AM=A\cdot M\subset M$ the subgroup
spanned by the elements $am$, \,$a\in A$, \,$m\in M$ (as usual).
 Now let $\R$ be a topological ring, $\C$ be a left $\R$\+contramodule,
and $\A\subset\R$ be a closed additive subgroup.
 Then the subgroup $\A\tim\C\subset\C$ is defined as the image of
the the composition of the natural inclusion $\A[[\C]]\hookrightarrow
\R[[\C]]$ and the contraaction map $\R[[\C]]\rarrow\C$.
 Clearly, one has $\A\C\subset\A\tim\C$ \,\cite[Section~2.10]{Pproperf}.

\subsection{Finiteness conditions on contramodules} \label{finiteness-subsecn}

 Notice that $\R[[X]]=\R[X]$ for any \emph{finite} set~$X$.
 An $\R$\+contramodule is said to be \emph{finitely generated} if it is
a quotient contramodule of a free $\R$\+contramodule $\R[[X]]$ spanned
by a finite set~$X$.
 An $\R$\+contramodule is said to be \emph{finitely presented} if it
can be presented as the cokernel of a morphism of free
$\R$\+contramodules $\R[[Y]]\rarrow\R[[X]]$ with finite sets $X$
and~$Y$.
 The forgetful functor $\R\Contra\rarrow\R\Modl$ restricts to
an equivalence between the full subcategory of finitely presented
$\R$\+contramodules in $\R\Contra$ and the full subcategory of finitely
presented $\R$\+modules in $\R\Modl$.
 In particular, the same functor provides an equivalence between
the full subcategories of finitely generated projective left
$\R$\+contramodules and finitely generated projective left
$\R$\+modules~\cite[Section~10]{PPT}.

Another finiteness property one can impose on an $\R$\+contramodule $\C$ can be defined in terms of the lattice of subcontramodules. As in any complete and cocomplete abelian category, subcontramodules of $\C$ form a complete lattice. More explicitly, if $\C_x\subset\C$ is a family of subcontramodules of $\C$ indexed by a set $X$ (i.e.\ the contraaction map $\R[[\C]]\rarrow\C$ restricts to $\R[[\C_x]]\rarrow\C_x$ for each $x\in X$), then the meet is just the set-theoretic intersection $\bigcap_{x\in X}\C_x\subset\C$, whereas the join is the image of the canonical map $\coprod_{x\in X}\C_x\rarrow\C$. In analogy with the notation for $\R$\+modules, we will denote the join by $\sum_{x\in X}\C_x$, and using the contraaction on $\C$, we can give a direct description in terms of elements,
\[ \sum_{x\in X}\C_x = \Biggl\{\, \sum_{x\in X}r_xc_x \;\Big|\; c_x\in\C_x \textrm{ and } \sum_{x\in X}r_xx\in\R[[X]] \,\Biggr\}. \]
If $X$ is finite, the situation is easier as $\R[[X]]=\R[X]$. In that case, $\sum_{x\in X}\C_x$ is computed as the finite sum of the underlying $\R$\+modules.
%It follows that the lattice of subcontramodules of a contramodule is modular, as is a lattice of submodules of a module.

Now we say that a contramodule $\C$ is \emph{lattice-finite} if for any family of subcontramodules $\C_x\subset\C$, $x\in X$, such that $\sum_{x\in X}\C_x=\C$, there exists a finite subset $F\subset X$ such that $\sum_{x\in F}\C_x=\C$. Note that any lattice-finite $\R$\+contramodule $\C$ is finitely generated as the contraaction map $\coprod_{c\in\C}\R=\R[[\C]]\rarrow\C$ is always a surjective homomorphism and by lattice-finiteness there exists a finite subset $F\subset\C$ such that the restriction to $\R[[F]]\rarrow\C$ is still surjective. If $\R$ is discrete (so that $\R\Contra=\R\Modl$), lattice-finiteness is well-known to be equivalent to being finitely generated, but for non-discrete topological rings, the converse implication might fail. In fact, a topological ring $\R$ itself is always finitely generated as an $\R$\+contramodule, but may not be lattice-finite. For example, this is the case if $\R$ has an infinite coproduct decomposition $\R=\coprod_{x\in X}\P_x$ in $\R\Contra$ (see Lemma~\ref{lattice-finite-proj-lem} below).

\subsection{Topologically semisimple and topologically perfect topological rings} \label{top-ss-and-perfect-subsecn}

 For any ring $R$, we denote by $H(R)\subset R$ the Jacobson radical
of~$R$.
 Given a topological ring $\R$, the \emph{topological Jacobson
radical} $\HH(\R)$ is defined as the intersection of all the open
maximal right ideals in~$\R$.
 The topological Jacobson radical $\HH(\R)$ is a closed two-sided ideal
in $\R$ containing the abstract Jacobson radical~$H(\R)$;
see~\cite[Section~1]{Gre}, \cite[Section~3.B]{IMR},
\cite[Section~7]{Pproperf}, or~\cite[Section~3]{BPS} for a discussion.

 By a \emph{semisimple} ring $S$ we mean a ring whose category of left
(equivalently, right) modules is semisimple; in other words, $S$ is
a semisimple left (equivalently, right) Artinian ring.
 Such rings are known as \emph{classically semisimple}.
 A topological ring $\S$ is called \emph{topologically semisimple} if
the abelian category of discrete right $\S$\+modules $\Discr\S$ is split
(equivalently, semisimple); this holds if and only if the abelian
category of left $\S$\+contramodules $\S\Contra$ is split (equivalently,
semisimple~\cite[Sections~2 and~6]{PS3}).
 The topologically semisimple topological rings are explicitly described
as the infinite topological products of the topological rings of
row-finite infinite matrices over skew-field~\cite[Theorem~3.10]{IMR},
\cite[Theorem~6.2]{PS3}.

 We refer to~\cite[Section~11]{Pextop}
and~\cite[Sections~2.11--2.12]{Pproperf} for a detailed
discussion of \emph{strongly closed subgroups} in topological
abelian groups and (in particular) strongly closed two-sided ideals
in topological rings.
 The point is that the quotient group of a topological group $\A$ by
a closed subgroup $\HH$ is always separated, but \emph{not} always
complete in the quotient topology.
 Even when the quotient group $\Q=\A/\HH$ is complete, the induced map
$\A[[X]]\rarrow\Q[[X]]$ need \emph{not} be surjective for an arbitrary
set~$X$ (in other words, the problem of lifting a zero-convergent
family of elements in $\Q$ to a zero-convergent family of elements in
$\A$ is not always solvable).
 When the quotient group $\Q=\A/\HH$ is complete and the map
$\A[[X]]\rarrow\Q[[X]]$ is surjective for all sets $X$, one says
that the closed subgroup $\HH\subset\A$ is strongly closed.

 We recall that an ideal $H$ in a ring $R$ is said to be \emph{left
T\+nilpotent} if, for every sequence of elements $a_1$, $a_2$,
$a_3$,~\dots~$\in H$ there exists an integer $m\ge1$ such that
the product $a_1\dotsm a_m$ vanishes in~$R$.
 An ideal $\HH$ in a topological ring $\R$ is said to be
\emph{topologically left T\+nilpotent} if, for every sequence of
elements $a_1$, $a_2$, $a_3$,~\dots~$\in\HH$, the sequence of
products $a_1$, $a_1a_2$, $a_1a_2a_3$,~\dots, $a_1\dotsm a_m$,~\dots\
converges to zero in the topology of~$\R$
\,\cite[Section~6]{Pproperf}, \cite[Section~7]{PS3}.
 A topological ring $\R$ is said to be \emph{topologically left
perfect} if its topological Jacobson radical $\HH$ is topologically left
T\+nilpotent and strongly closed in $\R$, and the topological
quotient ring $\S=\R/\HH$ is topologically
semisimple~\cite[Sections~10 and~14]{PS3}.

 Given a continuous homomorphism of topological rings $f\:\R\rarrow\S$,
any left $\S$\+contramodule $\D$ can be endowed with the left
$\R$\+contramodule structure induced via~$f$.
 The resulting exact, faithful, product-preserving functor taking
a left $\S$\+contramodule $\D$ to the left $\R$\+contramodule $\D$
is denoted by $f_\sharp\:\S\Contra\rarrow\R\Contra$ and called
the \emph{contrarestriction of scalars} with respect to~$f$.
 The functor~$f_\sharp$ has a left adjoint functor $f^\sharp\:
\R\Contra\rarrow\S\Contra$, which is called the \emph{contraextension
of scalars} \cite[Section~2.9]{Pproperf}.

 Let $\HH$ be a strongly closed two-sided ideal in a topological
ring $\R$, and let $\S=\R/\HH$ be the quotient ring endowed with
the quotient topology.
 Then the functor of contrarestriction of scalars $\S\Contra\rarrow
\R\Contra$ is fully faithful, and its essential image consists of
all the left $\R$\+contramodules $\D$ such that $\HH\tim\D=0$.
 The functor of contraextension of scalars $\R\Contra\rarrow
\S\Contra$ assigns to a left $\R$\+contramodule $\C$ the quotient
group $\C/(\HH\tim\C)$ endowed with the naturally induced
left $\S$\+contramodule structure~\cite[Section~2.12]{Pproperf}.

\subsection{Topologically agreeable additive categories} \label{agreeable-subsecn}

 An additive category $\sA$ with set-indexed coproducts is called
\emph{agreeable} if, for any object $M$ and family of objects
$(N_x\in\sA)_{x\in X}$, the natural map $\Hom_\sA(M,\>
\coprod_{x\in X}N_x)\allowbreak\rarrow\prod_{x\in X}\Hom_\sA(M,N_x)$
is injective.
 In particular, an additive category with products and coproducts is
agreeable if and only if the natural map from the coproduct to
the product $\coprod_{x\in X}N_x\rarrow\prod_{x\in X}N_x$ is
a monomorphism for every family of objects $N_x\in\sA$.
 A family of morphisms $f_x\:M\rarrow N_x$ in an agreeable category
$\sA$ is said to be \emph{summable} if it arises from a morphism
$f\:M\rarrow\coprod_{x\in X}N_x$ \,\cite{Cor}.

 A (complete, separated) \emph{right topological} additive category
$\sA$ is an additive category in which all the groups of morphisms
$\Hom_\sA(M,N)$ are endowed with complete, separated topologies
such that the composition maps are continuous and open
$\Hom_\sA(N,N)$\+submodules form a base of neighborhoods of zero
in $\Hom_\sA(M,N)$ for every pair of objects $M$, $N\in\sA$.
 One can show that any zero-convergent family of morphisms
$f_x\in\Hom_\sA(M,N)$ in a right topological additive category
is summable.
 A \emph{topologically agreeable} category $\sA$ is an agreeable
additive category endowed with a right topological additive category
structure for which the converse implication holds: any summable
family of morphisms $f_x\:M\rarrow N$ converges to zero in
the topology of $\Hom_\sA(M,N)$ \,\cite[Section~3]{PS3}.

 In particular, the abelian category of modules over an associative
ring $A$ can be endowed with a right topological category structure,
making it a topologically agreeable category, in several alternative
ways described in~\cite[Examples~3.7--3.8 and~3.10]{PS3}.
 Among these, the most natural approach is the one using
the \emph{finite topology} on the group of morphisms $\Hom_A(M,N)$
for any two left $A$\+modules $M$ and~$N$.
 A base of neighborhoods of zero in the finite topology consists
of the annihilators of finitely generated $A$\+submodules (or
equivalently, of finite subsets of elements) in~$M$.

 In particular, let $M$ be a left $A$\+module.
 Consider the ring $\R=\End_A(M)^\rop$ opposite to the endomorphism
ring of the $A$\+module~$M$; so $\R$ acts in $M$ on the right,
making $M$ an $A$\+$\R$\+bimodule.
 Then $\R$ is a complete, separated right linear topological ring
in the finite topology (where, once again, the annihilators of
finitely generated $A$\+submodules in $M$ form a base of neighborhoods
of zero)~\cite[Section~7.1]{PS1}.

 Before we finish these preliminaries, let us make a comment on our
notational conventions.
 Following the tradition of the preceding papers~\cite{Prev,PS1,
Pproperf,PS3,BPS} (going back to the book~\cite{Psemi}), we prefer to
deal with \emph{left} contramodules and discrete \emph{right} modules.
 Consequently, we have to consider \emph{right} linear topological
rings and produce them as the \emph{right} endomorphism rings
$\End_A(M)^\rop$ of \emph{left} modules~${}_AM$.
 For this reason, we will occasionally need a notation for morphisms
acting on objects (typically, modules) on the right.
 When we need to emphasize this aspect, we will denote our morphisms
by left-pointing arrows and write the morphism on the right-hand side
of the arrow; so $M\larrow M\,:\!r$ denotes the endomorphism of
a module $M$ corresponding to an element $r\in\R=\End_A(M)^\rop$.

\Section{Pontryagin Duality for Projective Contramodules} \label{pontryagin-secn}

Given a complete separated right linear topological ring $\R$, we establish in this section a duality between the category $\R\Contra_\proj$ of projective left $\R$\+contramodules and a suitable category of complete separated topological right $\R$-modules. This generalizes the well-known duality between finitely generated left and right projective modules over a discrete ring $R$, and allows us to translate certain properties of projective left $\R$-contramodules to corresponding properties of topological right $\R$\+modules and vice versa.

 In order to construct the corresponding functors, recall that $\R\Contra_\proj$ is a topologically agreeable
additive category~\cite[Remark~3.12]{PS3}; so the groups of morphisms
in it are endowed with natural topologies (see the discussion
in Subsection~\ref{agreeable-subsecn}).
 Thus, we have an additive functor $\Hom^\R({-},\R)\:(\R\Contra_\proj)^\sop\rarrow\Modcs\R$.
 Furthermore, the functor $\Hom^\R({-},\R)$ takes coproducts
of projective $\R$\+contramodules to topological products of
the complete separated topological right $\R$\+modules, by~\cite[Lemma~10.7]{PS3}, and it sends $\R$ viewed as a free left contramodule with one generator to $\R$ viewed as a complete separated topological right $\R$\+module.

Conversely, given any $M\in\Modr\R$, the group of morphisms $\Hom_\R(M,\R)$ has a natural structure of a left $\R$\+contramodule.
More in detail, suppose we have a collection of
%continuous
homomorphisms $f_x\:M\rarrow\R$, $x\in X$, and a zero-convergent family $\sum_{x\in X}r_xx\in\R[[X]]$. Then we define a homomorphism of $\R$\+modules $\sum_{x\in X}r_xf_x\:M\rarrow\R$ by the formula
\[
 \Biggl(\sum_{x\in X}r_xf_x\Biggr)(m) = \sum_{x\in X}r_x(f_x(m)) \in\R,
\]
where the latter sum is the topological limit of finite partial sums in $\R$.
If $\scrM\in\Modcs\R$ is a complete separated topological module and all $f_x\:\scrM\rarrow\R$ are continuous, it follows that $\sum_{x\in X}r_xf_x$ is also continuous.
 Indeed, for any open right ideal $\I\subset\R$, we find a finite subset
$F\subset X$ such that $r_x\in\I$ for each $x\in X\setminus F$.
 For each $x\in F$, there is an open right ideal $\J_x\subset\R$ such
that $r_x\J_x\subset\I$ and an open submodule $\scrN_x\subset\scrM$
such that $f_x(\scrN_x)\subset\J_x$.
 Then $\scrN=\bigcap_{x\in F}\scrN_x$ is an open submodule of $\scrM$
such that $\sum_{x\in X}r_xf_x(\scrN)\subset\I$, as required.
 Therefore, $\Hom_\R^\cont(\scrM,\R)$ is an $\R$-subcontramodule of
$\Hom_\R(\scrM,\R)$ and we have an additive functor
$\Hom_\R^\cont({-},\R)\:(\Modcs\R)^\sop\rarrow\R\Contra$.

Now we can state the duality in the following form.

\begin{thm} \label{pontryagin-thm}
Let $\R$ be a complete separated right linear topological ring and denote by $\Prod(\R)$ the full subcategory of $\Modcs\R$ formed by direct summands of products of copies of $\R$. Then there are mutually inverse equivalences of categories
\[ \Hom^\R({-},\R)\:(\R\Contra_\proj)^\sop \rightleftarrows \Prod(\R)\,:\!\Hom_\R^\cont({-},\R)^\sop. \]
\end{thm}

In order to prove the theorem, we first prove a sequence of three easy lemmas.

\begin{lem} \label{prodR-to-discrete-lem}
Let $X$ be a set and $N\in\Discr\R$ a discrete right $\R$\+module. Then the following assignment is bijective:
\begin{align*}
N[X]&\rarrow\Hom_\R^\cont(\R^X,N),\\
\sum_{x\in X}n_xx&\longmapsto
\Biggl((r_x)_{x\in X}\mapsto\sum_{x\in X}n_xr_x\Biggr).
\end{align*}
\end{lem} 

\begin{proof}
The injectivity being trivial, we prove that the assignment is surjective. Suppose that $f\:\R^X\rarrow N$ is a continuous homomorphism of right $\R$\+modules. Then $\ker(f)=f^{-1}(0)$ is an open submodule of $\R^X$ and, since $\R^X$ carries the product topology, there exists a finite subset $F\subset X$ such that $\R^{X\setminus F}\subset\ker(f)$. In particular, $f$ factors as $\R^X\rarrow\R^F\rarrow N$, where the first map is the canonical projection and the second map is given by a formal linear combination $\sum_{x\in F}n_xx\in N^F\subset N[X]$.
\end{proof}

When stated properly, an analogous result holds for all complete separated right $\R$\+modules rather than just discrete ones.

\begin{lem} \label{prodR-to-cs-lem}
Let $X$ be a set and $\scrN\in\Modcs\R$ a complete separated right $\R$\+module. Then the following assignment is bijective, where the infinite sum on the left is a formal linear combination whose coefficients converge to zero in $\scrN$, and the infinite sum on the right is the topological limit of finite partial sums in $\scrN$:
\begin{align*}
\scrN[[X]]&\rarrow\Hom_\R^\cont(\R^X,\scrN),\\
\sum_{x\in X}n_xx&\longmapsto
\Biggl((r_x)_{x\in X}\mapsto\sum_{x\in X}n_xr_x\Biggr).
\end{align*}
\end{lem} 

\begin{proof}
The topological module $\scrN$ is the inverse limit in $\Modcs\R$ of its discrete quotient modules $\scrN\!/\!\scrM$. Now note that the isomorphism from the statement is obtained as the inverse limit of the isomorphisms $(\scrN\!/\!\scrM)[X]\rarrow\Hom_\R^\cont(\R^X,\scrN\!/\!\scrM)$ from Lemma~\ref{prodR-to-discrete-lem}, using the canonical identifications $\varprojlim_{\scrM\subset\scrN}\Hom_\R^\cont(\R^X,\scrN\!/\!\scrM)=\Hom_\R^\cont(\R^X,\varprojlim_{\scrM\subset\scrN}\scrN\!/\!\scrM)=\Hom_\R^\cont(\R^X,\scrN)$.
\end{proof}

If we specialize the previous lemma to $\scrN=\R^Y$ for a set $Y$, we obtain that $\Hom_\R^\cont(\R^X,\R^Y)\simeq(\R^Y)[[X]]$. The last lemma provides a more convenient point of view at the right hand side of this isomorphism.
To this end, we denote by $\Mat_{Y\times X}(\R)$ the set of all row-zero-convergent (possibly infinite) matrices of elements of $\R$ with rows indexed by elements of $Y$ and columns by elements of $X$. In other words, an element of $\Mat_{Y\times X}(\R)$ is a family $(r_{yx}\in\R)_{y\in Y,x\in X}$ of elements of $\R$ such that for every $y\in Y$, the family of elements $(r_{yx})_{x\in X}$ converges to zero in the topology of $\R$.

\begin{lem} \label{YxX-matrices-lem}
For any sets $X$ and $Y$, we have a bijective assignment
\begin{align*}
(\R^Y)[[X]]&\rarrow\Mat_{Y\times X}(\R),\\
\sum_{x\in X}\big((r_{yx})_{y\in Y}\big)x&\longmapsto(r_{yx})_{y\in Y,x\in X}.
\end{align*}\end{lem}

\begin{proof}
If $\sum_{x\in X}(r_{yx})x\in(\R^Y)[[X]]$, then $(r_{yx})_{x\in X}$ is a zero-convergent family of elements of $\R$ for each $y\in Y$ since the product projections $\pi_y\:\R^Y\rarrow\R$ are continuous. It follows that the assignment is well defined and clearly it is injective.

To prove the surjectivity, consider a matrix $(r_{yx})\in\Mat_{Y\times X}(\R)$. We must show that the family $\big((r_{yx})_{y\in Y}\big)_{x\in X}$ converges to zero in $\R^Y$. Since $\R^Y$ carries the product topology, it has a base of neighborhoods of zero of the form
\[ \U_{F,\I} = \{ (s_y)_{y\in Y} \mid s_y\in\I \textrm{ for each } y\in F \}, \]
where $F\subset Y$ is a finite subset and $\I\subset\R$ is an open right ideal. Given such $F$ and $\I$, note that for each $y\in F$ there exists a finite subset $G_y\subset X$ such that $r_{yx}\in\I$ for each $x\in X\setminus G_y$. Then $(r_{yx})_{y\in Y}\in\U_{F,\I}$ for each $x\in X\setminus G$, where $G=\bigcup_{y\in F}G_y$.
\end{proof}

\begin{proof}[Proof of Theorem~\ref{pontryagin-thm}]
Let $X$ be a set. Then $\Hom^\R(\R[[X]],\R)$ canonically identifies with $\R^X\in\Modcs\R$ by~\cite[Lemma~10.7]{PS3}, and $\Hom_\R^\cont(\R^X,\R)$ canonically identifies with $\R[[X]]\in\R\Contra_\proj$ by Lemma~\ref{prodR-to-cs-lem}.

Since any $\R$\+contramodule homomorphism $\R[[Y]]\rarrow\R[[X]]$ is given by its values on each $y\in Y$, and such a value is an element of $\R[[X]]$, so an $X$-indexed zero-convergent family $(r_{yx})_{x\in X}$ of elements of $\R$, we have a canonical identification $\Hom^\R(\R[[Y]],\R[[X]])\simeq\Mat_{Y\times X}(\R)$. In fact, if $A=(r_{yx})\in\Mat_{Y\times X}(\R)$ and we view $Y$-indexed zero convergent families $\mathbf s=(s_y)_{y\in Y}$ from $\R[[Y]]$ as row vectors, the corresponding left $\R$\+contramodule homomorphism $\R[[Y]]\rarrow\R[[X]]$ is given by matrix multiplication $\mathbf s\longmapsto \mathbf sA=(\sum_{y\in Y} s_yr_{yx})_{x\in X}$, where the infinite sums are computed as topological limits of finite subsums in $\R$.

Similarly, $\Hom^\R(\R^X,\R^Y)\simeq\Mat_{Y\times X}(\R)$ thanks to Lemmas~\ref{prodR-to-cs-lem} and~\ref{YxX-matrices-lem} (applied to $\scrN=\R^Y$).
More explicitly, if $A=(r_{yx})\in\Mat_{Y\times X}(\R)$ and we view elements $\mathbf t=(t_x)_{x\in X}\in\R^X$ as column vectors, the corresponding continuous homomorphism of right $\R$\+modules is given by matrix multiplication $\mathbf t\longmapsto A\mathbf t$, where the infinite summation is interpreted in the same way as in the previous paragraph.

All in all, we have just proved that there is a duality between the category of free left $\R$-contramodules on one hand and the category of topological products of copies of the topological right $\R$\+module $\R$ on the other hand. Passing to the idempotent completions on both sides, we obtain the equivalence from the statement of Theorem~\ref{pontryagin-thm}.
\end{proof}

\begin{cor}
 Let\/ $\sA$ be an idempotent-complete topologically agreeable
additive category, let $M\in\sA$ be an object
and let $\R=\End_\sA(M)^\rop$ be the topological endomorphism ring.
Then
\[ \Hom_\sA(-,M)\:\Add(M)^\rop\rarrow\Prod(\R) \]
is an equivalence of categories, where $\Add(M)\subset\sA$ is the 
full subcategory of $\sA$ formed by direct summands of coproducts of copies of $M$ and
$\Prod(\R)$ the full subcategory of $\Modcs\R$ formed by direct summands of products of copies of $\R$.
\end{cor}

\begin{proof}
We just combine the equivalence from \cite[Theorem~3.14(iii)]{PS3} with the duality from Theorem~\ref{pontryagin-thm}.
\end{proof}

\Section{Discrete Local and Semiperfect Rings} \label{discrete-secn}

 In this section we recall the basic results concerning semiperfect
rings and their characterization.
 We use the book~\cite{AF} as the reference source.

 A nonzero ring $R$ is called \emph{local} if noninvertible elements
form an additive subgroup in~$R$, or equivalently, the unit element
of $R$ is not a sum of two noninvertible elements.
 A ring is local if and only if it has a unique maximal right ideal,
and if and only if the quotient ring of $R$ by its Jacobson radical is
a division ring~\cite[Proposition~15.15]{AF}.

\begin{lem} \label{projective-cover-of-simple}
 Let $R$ be a ring and $P$ be a projective right $R$\+module.
 Then $P$ is a projective cover of a simple right $R$\+module if and
only if the endomorphism ring of $P$ is a local ring.
 Moreover, if any one of these conditions holds, then $P$ is isomorphic
to the right $R$\+module $eR$ for some idempotent element $e\in R$.
\end{lem}

\begin{proof}
 This is a part of~\cite[Proposition~17.19]{AF}.
\end{proof}

\begin{cor} \label{local-idempotent-cor}
 Let $R$ be a ring with the Jacobson radical $H=H(R)$, and let $e\in R$
be an idempotent element.
 Then the following conditions are equivalent:
\begin{enumerate}
\item the ring $eRe$ is local;
\item the right $R$\+module $eR$ is a projective cover of a simple
right $R$\+module;
\item the left $R$\+module $Re$ is a projective cover of a simple
left $R$\+module;
\item $eR/eH$ is a simple right $R$\+module;
\item $Re/He$ is a simple left $R$\+module.
\end{enumerate}
 Moreover, if any one of the above five equivalent conditions holds,
then $eH$ is the unique maximal submodule in the right $R$\+module $eR$,
and $He$ is the unique maximal submodule in the left $R$\+module~$Re$.
\end{cor}

\begin{proof}
 Follows immediately from the preceding lemma, together with the facts
that, for any finitely generated projective right $R$\+module $P$,
the submodule $PH\subset P$ is superfluous, contains all the other
superfluous submodules of~$P$, and is equal to the intersection of
all maximal submodules of~$P$ \,\cite[Propositions~9.13, 9.18,
and~17.10]{AF}.
 (Cf.~\cite[Corollary~17.20]{AF}.)
\end{proof}

 A ring $R$ is called \emph{semiperfect} if its quotient ring $R/H$ by
its Jacobson radical $H$ is semisimple and every idempotent element in
$R/H$ can be lifted to an idempotent element in~$R$.
 A ring $R$ is semiperfect if and only if it admits a finite set of
orthogonal idempotents $e_1$,~\dots, $e_n\in R$ such that
$\sum_{i=1}^n e_i=1$ and $e_iRe_i$ is a local ring for every~$i$
\,\cite[Theorem~27.6]{AF}.
 In this case, the right $R$\+modules $e_iR/e_iH$ are simple, and
the semisimple right $R$\+module $R/H$ is isomorphic to
$\bigoplus_{i=1}^n e_iR/e_iH$.

 A ring $R$ is semiperfect if and only if any simple right
(equivalently, left) $R$\+module has a projective cover, and
if and only if every finitely generated right (equivalently, left)
$R$\+module has a projective cover~\cite[Theorem~27.6]{AF},
\cite[Theorem~3.6]{Fac}.

\begin{lem} \label{endomorphisms-discrete-semiperfect}
 Let $A$ be an associative ring and $M$ be a left $A$\+module.
 Let $R=\End_A(M)^\rop$ be the opposite ring to the ring of
endomorphisms of the $A$\+module $M$; so $M$ is an $A$\+$R$\+bimodule.
 Then the ring $R$ is semiperfect if and only if $M$ is a finite
direct sum of $A$\+modules with local endomorphism rings.
\end{lem}

\begin{proof}
 ``If'': assume that ${}_AM=\bigoplus_{i=1}^nM_i$, where the rings
$R_i=\End_A(M_i)^\rop$ are local.
 Denote by $e_i\in R$ the projector onto the direct summand $M_i$
in~$M$; then $e_i\in R$ and the ring $R_i$ is isomorphic to $e_iRe_i$.
 Since $e_1$,~\dots, $e_n$ are orthogonal idempotents in $R$ and
$\sum_{i=1}^n e_i=1$, the assertion follows.

 ``Only if'': assume that $R$ is a semiperfect ring, and let
$e_1$,~\dots, $e_n\in R$ be a set of orthogonal idempotents with
$\sum_{i=1}^n e_i=1$ such that the rings $e_iRe_i$ are local.
 Then the images $M_i$ of the $A$\+module endomorphisms
$e_i\:M\rarrow M$ form a direct sum decomposition of $M$, that is
$M=\bigoplus_{i=1}^nM_i$; and the endomorphism ring of the $A$\+module
$M_i$ is isomorphic to $e_iRe_i$.
\end{proof}

 Let $M=\bigoplus_{i=1}^n M_i$ be a left $A$\+module decomposed into
a finite direct sum of $A$\+modules~$M_i$.
 Then elements of the ring $R=\End_A(M)^\rop$ can be represented by
matrices $(r_{j,i})_{j,i=1}^n$ whose entries are $A$\+module morphisms
$M_i\larrow M_j\,:\!r_{j,i}$.

\begin{prop} \label{discrete-jacobson-as-matrices-of-nonisos}
 Let $A$ be an associative ring and $M$ be a left $A$\+module
decomposed into a finite direct sum $M=\bigoplus_{i=1}^n M_i$ of
modules with local endomorphism rings $R_i=\End_A(M_i)^\rop$.
 Then the Jacobson radical $H\subset R$ of the ring
$R=\End_A(M)^\rop$ is the set of all matrices $(h_{j,i})_{j,i=1}^n$
such that, for every pair of indices $j$ and $i$, the morphism
$M_i\larrow M_j\,:\!h_{j,i}$ is \emph{not} an isomorphism.
\end{prop}

\begin{proof}
 Denote temporarily by $H'\subset R$ the subset of all matrices of
nonisomorphisms.
 Using the assumption that the rings $R_i$ are local, one can easily
check that $H'$ is a two-sided ideal in $R$ and the quotient ring
$S=R/H'$ is semisimple.
 In fact, the relation of being isomorphic $A$\+modules
is an equivalence relation on
the set of all $A$\+modules $M_i$, \ $1\le i\le n$; let us consider
the equivalence (isomorphism) classes of these modules.
 If $m$~is the number of such isomorphism classes,
then $S$ is a direct product of $m$~simple rings.
 The latter are the rings of matrices over division rings (the residue
skew-fields of the rings~$R_i$) with the sizes of the matrices equal
to the cardinalities of the isomorphism classes of the modules~$M_i$
(cf.\ the proof of Proposition~\ref{matrices-of-noninvertibles} below).

 Since the quotient ring $R/H'$ is semisimple, it follows that
$H\subset H'$.
 In order to show that $H=H'$, one observes that the quotient ring
$R/H$ is semisimple, since the ring $R$ is semiperfect by
Lemma~\ref{endomorphisms-discrete-semiperfect}.
 Furthermore, following the proof of
Lemma~\ref{endomorphisms-discrete-semiperfect} and the discussion
preceding its formulation, the semisimple right $R$\+module $R/H$ is
a direct sum of $n$~simple modules.
 On the other hand, the discussion in the previous paragraph implies
that the semisimple right $R$\+module $R/H'$ is also a direct sum of
$n$~simple modules.
 Hence the natural surjective map $R/H\rarrow R/H'$ is an isomorphism,
and we can conclude that $H=H'$.
\end{proof}

\Section{Topologically Semiperfect Topological Rings}
\label{topologically-semiperfect-secn}

 Let $\R$ be a topological ring.
 Recall that the coproduct of a family of $\R$\+con\-tra\-mod\-ules
$\C_\alpha$, taken in the category $\R\Contra$, is denoted by
$\coprod_\alpha\C_\alpha=\coprod_\alpha^{\R\Contra}\C_\alpha$.
 Let us emphasize that the forgetful functor $\R\Contra\rarrow\R\Modl$
does \emph{not} usually preserve coproducts.
 As mentioned in Subsection~\ref{contramodules-subsecn}, the group of all
morphisms $\P\rarrow\Q$ in the category $\R\Contra$ is denoted by
$\Hom^\R(\P,\Q)$.

\begin{thm} \label{topologically-semiperfect-rings}
 Let\/ $\R$ be a topological ring.
 Then the following conditions are equivalent:
\begin{enumerate}
\item the free left\/ $\R$\+contramodule with one generator\/ $\R\in
\R\Contra$ decomposes as a coproduct of\/ $\R$\+contramodules with
local endomorphism rings (in the category\/ $\R\Contra$);
\item the right\/ $\R$\+module\/ $\R$, viewed as a topological
right\/ $\R$\+module, decomposes as a direct product of
topological\/ $\R$\+modules with local endomorphism rings (with
the product topology on the direct product);
\item there exists a set $Z$ and a zero-convergent family of elements\/
$\mathbf e=(e_z\in\nobreak\R)_{z\in Z}\allowbreak\in\R[[Z]]$ such that
$(e_z\in\R)_{z\in Z}$ is a family of pairwise orthogonal idempotents,
$\sum_{z\in Z}e_z=1$ in\/ $\R$, and $e_z\R e_z$ is a local ring for
every $z\in Z$.
\end{enumerate}
\end{thm}

 Let us make several comments concerning the formulation of the theorem.
 First of all, when speaking of some rings being local, we consider
them as abstract rings, irrespectively of any topology.
 Notice that, in any separated right linear topology on a local ring,
the maximal ideal is open (since there exists a proper open right ideal
and it is contained in the maximal ideal).

 Furthermore, any direct summand $\P$ of the $\R$\+contramodule $\R$ is
a finitely presented (in fact, finitely generated projective)
$\R$\+contramodule; hence the endomorphism rings of $\P$ in
the categories $\R\Contra$ and $\R\Modl$ agree.
 Similarly, any direct summand $\Q$ of the topological right
$\R$\+module $\R$ has the form $\Q=e\R$ for some idempotent element
$e\in\R$, with the direct summand topology on $e\R$; hence
the endomorphism ring of the topological right $\R$\+module $\Q$
agrees with the endomorphism ring of the abstract right
$\R$\+module $\Q$ (both are equal to $e\R e$);
cf.\ Subsection~\ref{discrete-and-cs-modules-subsecn}.

 The infinite sum $\sum_{z\in Z}e_z\in\R$ is understood as the limit of
finite partial sums in the topology of~$\R$.
 It is well-defined, because the family of elements $e_z\in\R$ is
zero-convergent by assumption (and the topological ring $\R$ is
complete and separated).

\begin{proof}[Proof of Theorem~\ref{topologically-semiperfect-rings}]
 (1)\,$\Longleftrightarrow$\,(2)
 This is an immediate consequence of Theorem~\ref{pontryagin-thm}.

 %(1)\,$\Longrightarrow$\,(2)
 %The desired implication follows from the fact that the full subcategory
%of projective objects in $\R\Contra$ is a topologically agreeable
%additive category~\cite[Remark~3.12]{PS3}; so the groups of morphisms
%in it are endowed with natural topologies (see the discussion
%in Section~\ref{prelim-secn}).
 %Thus an isomorphism of projective left $\R$\+contramodules
%$\R\simeq\coprod_{z\in Z}^{\R\Contra}\P_z$ is transformed by
%the functor $\Hom^\R({-},\R)$ into an isomorphism of
%topological right $\R$\+modules.
 %Furthermore, the functor $\Hom^\R({-},\R)$ takes coproducts
%of projective $\R$\+contramodules to topological products of
%the topological groups of morphisms, by~\cite[Lemma~10.7]{PS3}.
%
 %Finally, $\P_z$ is a direct summand of the left $\R$\+(contra)module
%$\R$, so we have $\P_z\simeq\R e_z$ for some idempotent element
%$e_z\in\R$.
 %Then $\Hom^\R(\P_z,\P_z)=\Hom_\R(\P_z,\P_z)\simeq(e_z\R e_z)^\rop
%\simeq\Hom_{\R^\rop}(\Q_z,\Q_z)^\rop$, where
%$\Q_z=\Hom^\R(\P_z,\R)=\Hom_\R(\P_z,\R)$.
 %So the endomorphism ring of $\Q_z$ is local if and only if
%the endomorphism ring of $\P_z$~is.

 (2)\,$\Longrightarrow$\,(3)
 Suppose that we are given an isomorphism of topological right
$\R$\+modules $\R\simeq\prod_{z\in Z}\Q_z$, with the product
topology on the right-hand side.
 Then, for any fixed $z\in Z$, we have a direct sum decomposition
$\R\simeq\Q_z\oplus\prod_{y\in Z}^{y\ne z}\Q_y$ of the (topological)
right $\R$\+module~$\R$.
 Let $e_z\in\R$ be the idempotent element such that the projector
$\R\rarrow\Q_z\rarrow\R$ (which is a right $\R$\+module morphism)
is given by the left multiplication with~$e_z$; so $\Q_z\simeq e_z\R$.

 Choosing two elements $z\ne w\in Z$ and considering the direct
sum decomposition $\R\simeq\Q_z\oplus\Q_w\oplus
\prod_{y\in Z}^{z\ne y\ne w}\Q_y$, one easily shows that $e_z$
and~$e_w$ are orthogonal idempotents in~$\R$.

 So we have an isomorphism of topological right $\R$\+modules
$\R\simeq\prod_{z\in Z}e_z\R$.
 Let $\I\subset\R$ be an open right ideal.
 Then, by the definition of the product topology, there exists
a subset $Y\subset Z$ with a finite complement $Z\setminus Y$
such that $\prod_{y\in Y}e_y\R\subset\I$.
 Hence $e_y\in\I$ for all $y\in Y$, and we have shown that
the family of elements $(e_z)_{z\in Z}$ converges to zero in~$\R$.

 Finally, in the notation of the previous paragraph we have
$\R\simeq\bigoplus_{z\in Z\setminus Y}e_z\R\oplus
\prod_{y\in Y}e_y\R$.
 Under this direct sum decomposition, the element $1\in\R$
corresponds to the element $\sum_{z\in Z\setminus Y}e_z+f
\in\bigoplus_{z\in Z\setminus Y}e_z\R\oplus\prod_{y\in Y}e_y\R$,
with $e_z\in e_z\R$ and $f\in\prod_{y\in Y}e_y\R\subset\I$.
 Hence $1-\sum_{z\in Z\setminus Y}e_z\in\I$, and we can conclude
that $\sum_{z\in Z}e_z=1$ in~$\R$.

 (3)\,$\Longrightarrow$\,(1)
 For any idempotent element $e\in\R$, the left $\R$\+submodule
$\R e\subset\R$ is a subcontramodule, and in fact naturally a direct
summand of $\R$ in the category $\R\Contra$.
 Hence, given a family of idempotent elements
$\mathbf e=(e_z\in\R)_{z\in Z}$, the contramodule
$\coprod_{z\in Z}^{\R\Contra}\R e_z$ is a direct summand of the free
$\R$\+contramodule $\coprod_{z\in Z}^{\R\Contra}\R\simeq\R[[Z]]$.
 It follows that, viewed as a subcontramodule in $\R[[Z]]$,
the coproduct $\coprod_{z\in Z}\R e_z$ is the set of all elements
$(r_z\in\R e_z)_{z\in Z}$ such that $(r_z\in\R)_{z\in Z}$ is
a zero-convergent family of elements in~$\R$.

 Now the map $f\:\coprod_{z\in Z}\R e_z\rarrow\R$ taking an element
$\mathbf r=(r_z\in\R e_z)_{z\in Z}$ to the element $f(\mathbf r)=
\sum_{z\in Z}r_z\in\R$ is an $\R$\+contramodule morphism (in fact,
a restriction of the natural $\R$\+contramodule morphism
$\R[[Z]]\rarrow\R$, which can be similarly constructed).
 Assuming that $\mathbf e\in\R[[Z]]$, a map $g\:\R\rarrow
\coprod_{z\in Z}\R e_z$ can be defined by the rule
$g(r)=(re_z)_{z\in Z}$ for every $r\in\R$.
 Assuming further that $\sum_{z\in Z}e_z=1$, the composition
$fg\:\R\rarrow\R$ is the identity map.
 Assuming that the idempotents~$e_z$ are pairwise orthogonal,
the composition $gf\:\coprod_{z\in Z}\R e_z\rarrow
\coprod_{z\in Z}\R e_z$ is the identity map.
 So $f$~is an isomorphism of $\R$\+contramodules.
 Finally, $\Hom^\R(\R e_z,\R e_z)\simeq(e_z\R e_z)^\rop$ is
a local ring by yet another assumption in~(3).
\end{proof}

 We will say that a topological ring $\R$ is \emph{topologically
semiperfect} if it satisfies the equivalent conditions of
Theorem~\ref{topologically-semiperfect-rings}.

\begin{prop} \label{semiperfect-decomposition}
 Let\/ $\sA$ be an idempotent-complete topologically agreeable
additive category, and let $M\in\sA$ be an object.
 Then the topological ring\/ $\R=\End_\sA(M)^\rop$ is topologically
semiperfect if and only if the object $M$ can be decomposed as
a coproduct of objects with local endomorphism rings.
\end{prop}

\begin{proof}
 Compare condition~(3) in Theorem~\ref{topologically-semiperfect-rings}
with~\cite[Lemma~10.10]{PS3}.
 Alternatively, compare condition~(1) in
Theorem~\ref{topologically-semiperfect-rings}
with~\cite[Theorem~3.14(iii)]{PS3}.
\end{proof}

\SectionSpecialChar{Structural Properties of Topologically Semiperfect
Topological Rings}%
{Structural Properties of Topologically Semiperfect Topological Rings}

 The notation $H(R)$ for the Jacobson radical of a ring $R$ was
introduced in Subsection~\ref{top-ss-and-perfect-subsecn} and already used
in Section~\ref{discrete-secn}.
 We denote the topological Jacobson radical of a topological ring
$\R$ by $\HH(\R)$; see Subsection~\ref{top-ss-and-perfect-subsecn}
for a brief discussion with references.
 The definition of the \emph{finite topology} on the endomorphism ring
of a module can be found in Subsection~\ref{agreeable-subsecn}.

\begin{lem} \label{no-idempotents-in-topological-Jacobson}
 Let\/ $\R$ be a topological ring and $e\in\HH(\R)$ be an element.
 Then the equation $e^2=e$ implies $e=0$.
\end{lem}

\begin{proof}
 Following~\cite[Lemma~3.8]{BPS}, for any element $h\in\HH(\R)$,
the right multiplication with $1-h$ is an injective map
$\R\rarrow\R$.
 Hence the conditions $e\in\HH(\R)$ and $e(1-e)=0$ imply $e=0$.
\end{proof}

 Let $A$ be an associative ring and $i\:M\rarrow N$ be a morphism
of left $A$\+modules.
 One says that $i$~is a \emph{locally split monomorphism}
(see~\cite[Introduction and Section~4]{BPS} for a historical discussion
with references) if, for any finite set of elements $x_1$,~\dots,
$x_m\in M$, there exists an $A$\+module morphism $g\:N\rarrow M$
such that $gi(x_j)=x_j$ for all $j=1$,~\dots,~$m$.
 Clearly, any locally split monomorphism of $A$\+modules is
an injective map.

\begin{lem} \label{topological-Jacobson-via-locally-split-monos}
 Let $A$ be an associative ring, $M$ be a left $A$\+module, and\/
$\R=\End_A(M)^\rop$ be (the opposite ring to) the endomorphism ring
of $M$, endowed with the finite topology.
 Then an element $h\in\R$ belongs to the topological Jacobson
radical\/ $\HH(\R)\subset\R$ if and only if, for every element
$r\in\R$, the $A$\+module morphism\/ $M\larrow M\,:\!(1-hr)$
is a locally split monomorphism.
\end{lem}

\begin{proof}
 Following~\cite[Lemma~7.2(iii)]{Pproperf}
or~\cite[Proposition 1.7]{Gre}, an element $h\in\R$
belongs to $\HH(\R)$ if and only if, for every $r\in\R$ and
every open right ideal $\I\subset\R$, one has $(1-hr)\R+\I=\R$.
 Since the annihilators of finitely generated submodules of $M$
form a base of neighborhoods of zero in $\R$, one can assume that
$\I$ is such an annihilator.
 Let $E\subset M$ be a finitely generated submodule and
$\I=\Hom_A(M/E,M)\subset\R$ be its annihilator.
 Then two elements of $\R$ differ by an element from $\I$ if and only
if the related two endomorphisms of the $A$\+module $M$ agree in
the restriction to~$E$.
 One easily concludes that the equation $(1-hr)\R+\I=\R$ holds for
all $\I$ if and only if $M\larrow M\,:\!(1-hr)$ is a locally split
monomorphism of $A$\+modules.
\end{proof}

 Let $\sA$ be a topologically agreeable additive category and
$M=\coprod_{z\in Z}M_z$ be an object of $\sA$ decomposed into
a coproduct indexed by a set~$Z$.
 Then elements of the topological ring $\R=\End_\sA(M)^\rop$ can be
represented by matrices $(r_{w,z})_{w,z\in Z}$ whose entries are
morphisms $M_z\larrow M_w\,:\!r_{w,z}$.
 More precisely, the ring $\R$ can be described as the ring of all
\emph{row-summable} matrices of this form~\cite[first paragraph of
the proof of ``only if'' implication in Theorem~10.4]{PS3}.
 (We refer to Subsection~\ref{agreeable-subsecn} for the background on
topologically agreeable categories and summable families of morphisms.)

\begin{prop} \label{matrices-of-noninvertibles}
 Let\/ $\sA$ be a topologically agreeable additive category and
$M=\coprod_{x\in Z}M_z$ be an object of\/ $\sA$ decomposed into
a coproduct of objects with local endomorphism rings\/
$\R_z=\End_\sA(M)^\rop$.
 Consider the subset\/ $\HH'\subset\R$ consisting of all the matrices
$(h_{w,z})_{w,z\in Z}$ such that, for all $w$, $z\in Z$, the morphism
$M_z\larrow M_w\,:\!h_{w,z}$ is \emph{not} an isomorphism.
 Then\/ $\HH'$ is a strongly closed two-sided ideal in\/ $\R$,
and the quotient ring\/ $\S=\R/\HH'$ is topologically semisimple in
the quotient topology.
\end{prop}

\begin{proof}
 We follow the arguments in the proof of the ``only if'' part
of~\cite[Theorem~10.4]{PS3} in~\cite[Section~10]{PS3}.
 Denote by $X$ the set of all isomorphism classes of the objects
$M_z\in\sA$.
 For every element $x\in X$, let $Y_x\subset Z$ denote the full
preimage of the element~$x$ under the natural surjective map
$Z\rarrow X$ assigning to an index $z\in Z$ the isomorphism class
of the object~$M_z$.
 So, given $z$ and $w\in Z$, we have $M_z\simeq M_w$ if and only if
there exists $x\in X$ such that $z$, $w\in Y_x$.

 Given $z$ and $w\in Z$, consider the topological group of morphisms
$\R_{w,z}=\Hom_\sA(M_w,M_z)$, and denote by $\HH'_{w,z}\subset\R_{w,z}$
the subset of all nonisomorphisms.
 By assumption, $\HH'_{z,z}$ is a two-sided ideal in~$\R_{z,z}$;
following the discussion after the formulation of
Theorem~\ref{topologically-semiperfect-rings}, it is
an open two-sided ideal.
 Similarly to the argument in~\cite{PS3}, it follows that $\HH'_{w,z}$
is an open subgroup in $\R_{w,z}$ for all $z$, $w\in Z$ (in fact,
one has $\HH'_{w,z}=\R_{w,z}$ when $M_z$ and $M_w$ are not isomorphic).
 As in~\cite{PS3}, one concludes that $\HH'$ is a closed subgroup
in~$\R$, and further that $\HH'$ is a two-sided ideal.

 Given $y\in Z$, denote by $D_y$ the discrete skew-field
$\R_{y,y}/\HH'_{y,y}$.
 Given $x\in X$, we choose isomorphisms between all the objects $M_y$,
\,$y\in Y_x$, in a compatible way, and put $D_x=D_y$.
 Similarly to the argument in~\cite{PS3}, the quotient ring
$\S=\R/\HH'$ with its quotient topology is described as the topological
product over $x\in X$ of the topological rings of $Y_x$\+sized
row-finite matrices with the entries in~$D_x$.
 By~\cite[Theorem~6.2(4)]{PS3}, the topological ring $\S$ is
topologically semisimple.

 Finally, the construction of a continuous section $s\:\S\rarrow\R$
as in the next-to-last paragraph of~\cite[Theorem~10.4, proof of
``only if'']{PS3} shows that the subgroup $\HH'\subset\R$ is
strongly closed.
\end{proof}

\begin{thm} \label{topologically-semiperfect-structural-properties}
 Let\/ $\R$ be a topologically semiperfect topological ring.
 Then the following properties hold:
\begin{enumerate}
\renewcommand{\theenumi}{\alph{enumi}}
\item the topological Jacobson radical\/ $\HH=\HH(\R)$ is strongly
closed in\/~$\R$;
\item the topological quotient ring\/ $\R/\HH$ is topologically
semisimple.
\end{enumerate}
\end{thm}

\begin{proof}
 By~\cite[Corollary~4.4]{PS3}, for any topological ring $\R$ there
exists an associative ring $A$ and a left $A$\+module $M$ such that\/
$\R$ is isomorphic, as a topological ring, to the endomorphism ring
$\End_A(M)^\rop$ endowed with the finite topology.
 By Proposition~\ref{semiperfect-decomposition}, if the topological
ring $\R$ is topologically semiperfect, then the left $A$\+module $M$
decomposes into a direct sum $M=\bigoplus_{z\in Z}M_z$ of left
$A$\+modules $M_z$ with local endomorphism rings
$\R_z=\End_A(M_z)^\rop$.
 
 Now the construction of Proposition~\ref{matrices-of-noninvertibles}
provides a strongly closed two-sided ideal $\HH'\subset\R$ with
a topologically semisimple topological quotient ring $\S=\R/\HH'$.
 Let us show that $\HH=\HH'$; this would obviously imply both
assertions of the theorem.

 First of all, every nonzero element of $\S$ acts nontrivially on some
simple discrete right $\S$\+module (since $\S$ is topologically
semisimple); see~\cite[Theorem~6.2(2)]{PS3}.
 Hence the ideal $\HH'\subset\R$ is the intersection of
the annihilators of those simple discrete right $\R$\+modules on
which $\R$ acts through~$\S$.
 Since $\HH$ is the intersection of the annihilators of all simple
discrete right $\R$\+modules~\cite[Lemma~7.2(ii)]{Pproperf}, it
follows that $\HH\subset\HH'$.

 To prove the inverse inclusion, we use
Lemma~\ref{topological-Jacobson-via-locally-split-monos}.
 It suffices to show that, for every $h\in\HH'$, the endomorphism
$M\larrow M\,:\!(1-h)$ is a locally split monomorphism of $A$\+modules.
 Let $E\subset M$ be a finitely generated submodule.
 Then there exists a finite subset $Z_0\subset Z$ such that
$E\subset\bigoplus_{z\in Z_0}M_z\subset M$.
 Put $K=\bigoplus_{z\in Z_0}M_z$ and
$L=\bigoplus_{w\in Z\setminus Z_0}M_w$; so $M=K\oplus L$.
 Consider the direct summand inclusion $M\hookleftarrow K$ and
the direct summand projection $K\twoheadleftarrow M$ along~$L$.
 We are interested in the composition $M\overset{1-h}\larrow M
\hookleftarrow K$; it suffices to show that this composition
$M\larrow K$ is a split monomorphism of $A$\+modules
(cf.~\cite[Lemma~4.4]{BPS}).
 For this purpose, we will check that the composition
$K\twoheadleftarrow M\overset{1-h}\larrow M\hookleftarrow K$ is
an isomorphism.

 Consider the endomorphism ring $R=\End_A(K)^\rop$.
 By Lemma~\ref{endomorphisms-discrete-semiperfect}, the ring $R$ is
semiperfect.
 Proposition~\ref{discrete-jacobson-as-matrices-of-nonisos} computes
the Jacobson radical $H(R)$ as the set of all $n\times n$ matrices
of nonisomorphisms $t=(t_{j,i})_{j,i\in Z_0}$, where $M_i\larrow M_j
\,:\!t_{j,i}$ and $n$~is the cardinality of~$Z_0$.
 Now the composition $K\twoheadleftarrow M\overset h\larrow M
\hookleftarrow K$ is such a matrix of nonisomorphisms (since
$h\in\HH'$).
 It remains to recall that the element $1-t$ is invertible in $R$
for every $t\in H(R)$.
\end{proof}

\begin{rem}
 Following~\cite[Section~10]{PS3}, a topological ring $\R$ is called
\emph{topologically left perfect} if its topological Jacobson radical
$\HH=\HH(\R)$ is a topologically left T\+nilpotent strongly closed ideal
in $\R$ and the topological quotient ring $\S=\R/\HH$ is topologically
semisimple.
 Any topological ring $\R$ can be obtained as the endomorphism ring of
a module $M$, with the finite topology on the endomorphism
ring~\cite[Corollary~4.4]{PS3}.
 Hence, comparing~\cite[Theorem~10.4]{PS3} with
Proposition~\ref{semiperfect-decomposition} above, one can see that
any topologically left perfect topological ring is topologically
semiperfect (as any perfect decomposition of a module is a decomposition
into a direct sum of modules with local endomorphism rings).
 Thus it follows from
Theorem~\ref{topologically-semiperfect-structural-properties} that
a topological ring $\R$ is topologically left perfect if and only if
it is topologically semiperfect and its topological Jacobson radical
$\HH$ is topologically left T\+nilpotent.

 A module $M$ over a ring $R$ is said to be \emph{coperfect} if any
descending chain of cyclic submodules in $R$ terminates, or
equivalently, any descending chain of finitely generated submodules
in $R$ terminates~\cite[Theorem~2]{Bj}.
 A topological ring $\R$ is called \emph{topologically right coperfect}
if all discrete right $\R$\+modules are coperfect.

 Any topologically left perfect topological ring is topologically right
coperfect~\cite[Theorem~14.4\,(iv)$\Rightarrow$(v)]{PS3}.
 By~\cite[Lemmas~7.4 and~7.5]{Pproperf}, the topological Jacobson
radical of any topologically right coperfect topological ring is
topologically left T\+nilpotent.
 Therefore, a topological ring $\R$ is topologically left perfect if and
only if it is simultaneously topologically semiperfect and topologically
right coperfect.

 It is conjectured in~\cite[Conjecture~14.3]{PS3} that a topological
ring is topologically left perfect if and only if it is topologically
right coperfect.
 By~\cite[Remark~14.9]{PS3}, this conjecture is equivalent to a positive
answer to a question of Angeleri H\"ugel and
Saor\'{\i}n~\cite[Question~1 in Section~2]{AS}.
 In view of the arguments above, it would be sufficient to prove that
any topologically right coperfect ring is topologically semiperfect
in order to establish the conjecture.
\end{rem}

\Section{Surjective Continuous Ring Homomorphisms}
\label{surjective-homomorphisms-secn}

 The results of this section are inspired by~\cite[Proposition~2.12
and Theorem~2.14]{Gre}.

 Let $R$ be an associative ring.
 An idempotent element $e\in R$ is said to be \emph{local} if the ring
$eRe$ is local.
 Since any nonzero quotient ring of a local ring is local, it follows
that any surjective ring homomorphism takes local idempotents to
local idempotents or zero.

\begin{prop} \label{continuous-image-stays-top-semiperfect}
 Let $g\:\R\rarrow\S$ be a surjective continuous homomorphism of
topological rings.
 Assume that the topological ring\/ $\R$ is topologically semiperfect.
 Then the topological ring\/ $\S$ is topologically semiperfect as well.
\end{prop}

\begin{proof}
 We use the characterization of topologically semiperfect topological
rings given by Theorem~\ref{topologically-semiperfect-rings}(3).
 Let $\mathbf e=(e_z\in\R)_{z\in Z}$ be a zero-convergent family of
pairwise orthogonal local idempotents in $\R$ such that
$\sum_{z\in Z}e_z=1$ in~$\R$.
 Denote by $W\subset Z$ the set of all indices $w\in Z$ for which
$g(e_w)\ne0$ in~$\S$.
 For every $w\in W$, put $f_w=g(e_w)\in\S$.
 Then $\mathbf f=(f_w\in\S)_{w\in W}$ is a family of pairwise
orthogonal local idempotents in~$\S$.
 Since the ring homomorphism~$g$ is continuous by assumption,
the family of elements $(f_w)_{w\in W}$ converges to zero in $\S$
and we have $\sum_{w\in W}f_w=\sum_{z\in Z}g(e_z)=
g(\sum_{z\in Z}e_z)=1$ in $\S$, as desired.
\end{proof}

 Let us state separately two particular cases of
Proposition~\ref{continuous-image-stays-top-semiperfect}
corresponding to the quotient topologies and the topology
coarsenings.

\begin{cor} \label{complete-quotient-stays-top-semiperfect}
 Let\/ $\R$ be a topologically semiperfect topological ring and\/
$\J\subset\R$ be a closed two-sided ideal such that the quotient
ring\/ $\S=\R/\J$ is complete in its quotient topology.
 Then the topological ring\/ $\S$ is topologically semiperfect. \qed
\end{cor}

\begin{cor} \label{complete-coarsening-stays-top-semiperfect}
 Let\/ $\R'$ and\/ $\R''$ be two topological ring structures on one
and the same ring\/ $\R'=R=\R''$ such that the topology of\/ $\R''$
is coarser than that of\/~$\R'$.
 Then the topological ring\/ $\R''$ is topologically semiperfect
whenever the topological ring\/ $\R'$ is.
 In other words, if $g\:\R'\rarrow\R''$ is a bijective continuous
homomorphism of topological rings and the topological ring\/ $\R'$
is topologically semiperfect, then the topological ring\/ $\R''$
is topologically semiperfect as well. \qed
\end{cor}

\begin{ex}
 The situation described in
Corollary~\ref{complete-coarsening-stays-top-semiperfect} happens
quite often.
 For example, let $A$ be a ring and $M$ be an infinitely generated
$A$\+module with a local endomorphism ring $R=\End_A(M)^\rop$.
 Denote by\/ $\R'$ the ring $R$ endowed with the discrete topology,
and let $\R''$ be the notation for the ring $R$ endowed with
the finite topology of the endomorphism ring.
 Then the topological ring $\R'$ is topologically semiperfect
(because any local ring is semiperfect as a discrete ring,
in the sense of Section~\ref{discrete-secn}), and the topological
ring $\R''$ is topologically semiperfect (because any complete,
separated right linear topology on a local ring makes it
topologically semiperfect, in the sense of
Section~\ref{topologically-semiperfect-secn}).

 To give a more specific example, let $R$ be a complete Noetherian
commutative local ring.
 Then the ring $R$ is semiperfect in its discrete topology, and it is
also topologically semiperfect in the adic topology of its maximal
ideal~$\mathfrak m$.
 In fact, the ring $R$ with the $\mathfrak m$\+adic topology is even
\emph{topologically perfect}, in the sense of~\cite{PS3} and
Section~\ref{top-ss-and-perfect-subsecn}, as the ideal~$\mathfrak m$
is topologically T\+nilpotent in its adic topology.
\end{ex}

\begin{rem}
 The converse assertion to
Corollary~\ref{complete-coarsening-stays-top-semiperfect} is not true:
the passage to a finer topology does \emph{not} preserve
the topological semiperfectness.
 For example, let $R$ be the product of an infinite family of fields.
 Then, endowed with the product topology (of the discrete topologies
on the fields), $R$ is topologically semiperfect, and in fact, even
topologically semisimple.
 But $R$ is not semiperfect as a discrete ring, since its Jacobson
radical $H(R)$ vanishes while the ring $R$ is not (classically)
semisimple.
 (Cf.~\cite[Proposition~2.15]{Gre}.)
\end{rem}

\begin{rem}
 It would be interesting to generalize
Proposition~\ref{continuous-image-stays-top-semiperfect} to
continuous, but not necessarily surjective ring homomorphisms~$g$
with a dense image.
 In particular, this might include the passage from a topological
ring to the completion of its quotient ring by a closed two-sided
ideal (cf.\ Section~\ref{top-ss-and-perfect-subsecn}) or to its
 completion with respect to a separated coarsening of its topology.
 However, a straightforward approach to such a generalization runs
into an obstacle in the case of a local ring $\R$ already.

 Let $\R$ be a topological ring that is local as an abstract ring,
and let $H=\HH\subset\R$ be its topological Jacobson radical, which
coincides with the abstract Jacobson radical in this case.
 Let $g\:\R\rarrow\S$ be a continuous homomorphism of topological
rings with a dense image.
 Then, assuming that $\S\ne0$, the right action of $\R$ on $\R/\HH$
extends uniquely to a discrete $\S$\+module structure, making $\R/\HH$
the unique discrete simple right $\S$\+module.
 Hence the topological Jacobson radical $\HH(\S)\subset\S$ coincides
with the topological closure of $g(\HH)$ in $\S$, and the topological
quotient ring $\S/\HH(\S)\simeq\R/\HH$ is a discrete division ring.
 But how can one prove that the \emph{abstract} Jacobson radical
of $\S$ coincides with the topological Jacobson radical?

 Generally, let $\S$ be a topological ring whose topological Jacobson
radical $\HH(\S)$ is an open ideal in $\S$ and the quotient ring
$\S/\HH(\S)$ is a division ring.
 Does it follow that $\S$ is a local ring, or in other words,
that $H(\S)=\HH(\S)$\,?
\end{rem}

\Section{Projective Covers of Finitely Generated Discrete Modules} \label{proj-cover-discrete}

 Let $R$ be an associative ring.
 A nonzero idempotent element $e\in R$ is said to be \emph{primitive}
if it cannot be presented as the sum of two nonzero orthogonal
idempotents.
 The definition of a \emph{local} idempotent was given in
Section~\ref{surjective-homomorphisms-secn}.
 Since a local ring contains no idempotents other than $0$ and~$1$,
it follows that any local idempotent is primitive.

 We start with a discussion of idempotents in a topologically
semisimple topological ring (in the sense of~\cite[Section~6]{PS3};
see Subsection~\ref{top-ss-and-perfect-subsecn}).

\begin{lem} \label{idempotents-in-topologically-semisimple-rings}
 Let\/ $\S$ be a topologically semisimple topological ring and
$g\in\S$ be an idempotent element.
 Then the following conditions are equivalent:
\begin{enumerate}
\item $g$ is a primitive idempotent;
\item $g$ is a local idempotent;
\item the right\/ $\S$\+module $g\S$ is (discrete and) simple;
\item the left\/ $\S$\+contramodule $\S g$ is simple.
\end{enumerate}
\end{lem}

\begin{proof}
 (1)\,$\Longleftrightarrow$\,(2)
 By~\cite[Theorem~6.2(3)]{PS3}, there exists a ring $A$ and
a semisimple left $A$\+module $M$ such that $\S$ is topologically
isomorphic to the ring $\End_A(M)^\rop$ endowed with the finite
topology.
 Now idempotent elements of $\S$ correspond to decompositions of
the $A$\+module $M$ into direct sums of two summands.
 It remains to observe that any direct summand of $M$ is semisimple,
and a semisimple module is indecomposable if and only if its
endomorphism ring is local.

 (1)\,$\Longrightarrow$\,(3)
 Following~\cite[proof of Theorem~6.2]{PS3}, in the context of
the previous paragraph one can assume that $A=\prod_{x\in X}D_x$
is a product of division rings and $M=\bigoplus_{x\in X}D_x^{(Y_x)}$
is a direct sum of vector spaces over these division rings, which is a semisimple $A$-module.
If $M\larrow M\,:\!g$ is a primitive idempotent, then $Mg$ is a simple $A$\+module and elements of the right ideal $g\S$ can be identified with $A$\+module homomorphisms $M\larrow Mg$. As $Mg$ is simple, the finite topology on $\Hom_A(Mg,M)$ is discrete, and so is the right ideal $g\S$. It is also not difficult to see that $g\S$ is a simple right $\S$\+module; see the description of such modules in~\cite[Remark~6.4]{PS3}.
% Now it is clear from the description of simple discrete right
%$\S$\+modules in~\cite[Remark~6.4]{PS3} that $g\S$ is such a module.

 (3)\,$\Longrightarrow$\,(2)
 The endomorphism ring $g\S g=\End_\S(g\S)$ of a simple
right $\S$\+module $g\S$ is a division ring, hence a local ring.

 (2)\,$\Longleftrightarrow$\,(4)
 By~\cite[Theorem~6.2(1)]{PS3}, the category of left
$\S$\+contramodules is Ab5 and semisimple; hence an object
$\C\in\S\Contra$ is simple if and only if its endomorphism ring
$\Hom^\S(\C,\C)$ is local.
 (In fact, in any split abelian category, an object is simple if and
only if it is indecomposable; the endomorphism ring of a decomposable
object cannot be local.)
 It remains to recall that for a finitely generated (in fact, cyclic)
projective left $\S$\+contramodule $\S g$ one has
$\Hom^\S(\S g,\S g)=\Hom_\S(\S g,\S g)\simeq(g\S g)^\rop$.
\end{proof}

\begin{lem} \label{simples-over-top-semisimple-via-idempotents}
 Let\/ $\S$ be a topologically semisimple topological ring and\/
$\mathbf g=(g_z)_{z\in Z}\in\S[[Z]]$ be a zero-convergent family
of pairwise orthogonal primitive/local idempotents such that\/
$\sum_{z\in Z}g_z=1$ in\/~$\S$.  Then
\begin{enumerate}
\renewcommand{\theenumi}{\alph{enumi}}
\item every simple discrete right\/ $\S$\+module has the form
$g_z\S$ for some $z\in Z$;
\item every simple left\/ $\S$\+contramodule has the form
$\S g_z$ for some $z\in Z$.
\end{enumerate}
\end{lem}

\begin{proof}
 As in the previous proof, we have $\S=\End_A(M)^\rop$, where
$A=\prod_{x\in X}D_x$ and $M=\bigoplus_{x\in X}D_x^{(Y_x)}$.
 The choice of a zero-convergent family of orthogonal primitive
idempotents $(g_z)_{z\in Z}$ with $\sum_{z\in Z}g_z=1$ is equivalent
to the choice of a decomposition of $M$ into a direct sum of simple
modules (i.~e., one-dimensional vector spaces).
 Both the assertions of the lemma can be now easily obtained
from~\cite[Remark~6.4]{PS3}.
\end{proof}

\begin{lem} \label{simple-contra-modules}
 Let\/ $\R$ be a topological ring and\/ $\C$ be a left\/
$\R$\+contramodule. 
 Then\/ $\C$ is simple (as an object of\/ $\R\Contra$) if and only if
the underlying left\/ $\R$\+module of\/ $\C$ is simple (as an object
of\/ $\R\Modl$). 
\end{lem}

\begin{proof}
 One observes that a module is simple if and only it contains
no proper nonzero \emph{cyclic} submodules.
 Then it remains to use~\cite[Lemma~3.4]{BPS}, which tells that
any cyclic submodule of a contramodule is a subcontramodule. 
\end{proof}

Next we adapt general facts about projective covers to the situation
of discrete modules over a topological ring.
A minor variation of the argument in~\cite[Theorem~17.9]{AF} shows us that the projective covers of discrete simples, when they exist, generate all discrete modules in $\Modr\R$.

\begin{lem} \label{projective-covers-of-discrete-simples-generate}
 Let\/ $\R$ be a topological ring and $X$ be a set indexing
all isomorphism classes of simple discrete right $\R$\+modules.
Assume that for each $x\in X$, a projective cover
$P_x\rarrow S_x$ of the corresponding discrete simple module $S_x$
exists in the category\/ $\Modr\R$. Then each $M\in\Discr\R$ 
admits a surjective homomorphism of right $\R$\+modules of the form
$\bigoplus_{i\in I}P_{x_i}\rarrow M$ with $x_i\in X$ for each $i\in I$.
\end{lem}

\begin{proof}
Let $M\in\Discr\R$ and $N\subset M$ be the unique largest $\R$\+submodule of $M$ generated by the set $\{P_x\mid x\in X\}$. We must prove that $M=N$. If not, there exists a nonzero finitely generated submodule $F/N\subset M/N$ which in turn has a maximal submodule $G/N\subset F/N$. So the subquotient $F/G$ of $M$ is simple and discrete and, hence, admits by assumption a projective cover $p\:P_x\rarrow F/G$ in $\Modr\R$ for some $x\in X$. As $P_x$ is projective in $\Modr\R$, the map $p$ lifts to a homomorphism of $\R$\+modules $f\:P_x\rarrow M$ whose image is not contained in $N$ by construction. This yields the desired contradiction and concludes the proof.
\end{proof}

Now we can adapt \cite[Theorem 27.6\,(c)$\Rightarrow$(d)]{AF}, which classically says that over a (discrete) ring, projective covers of all finitely generated modules exist provided that projective covers of all simple modules exist.

\begin{lem} \label{discrete-fg-projective-covers-provisional}
Let\/ $\R$ be a topological ring such that each
simple discrete right $\R$\+module admits a projective cover
in the category\/ $\Modr\R$.
Then in fact any finitely generated discrete right\/ $\R$\+module
has a projective cover in\/ $\Modr\R$.
\end{lem}

\begin{proof}
%The argument is modeled on the argument
%for~\cite[Theorem 27.6\,(c)$\Rightarrow$(d)]{AF} and uses 
%Lemmas~\ref{projective-covers-of-discrete-simples-generate}.
Suppose $M\in\Discr\R$ is finitely generated. Then there is a surjective homomorphism of right $\R$\+modules $P := P_{x_1}\oplus\cdots\oplus P_{x_n}\rarrow M$ for some $n\ge 0$. Here, $P_{x_i}$ are projective covers of certain simple discrete $\R$\+modules as in the statement of Lemma~\ref{projective-covers-of-discrete-simples-generate}.
Denoting $H=H(\R)$ the Jacobson radical of $\R$ as usual, this homomorphism induces a surjective homomorphism $P/PH\rarrow M/MH$. Since $P/PH$ is semisimple (by
Lemma~\ref{projective-cover-of-simple} and Corollary~\ref{local-idempotent-cor}), so is $M/MH$. 
As a finite direct sum of covers is a cover (see~\cite[Remark~1.4.2]{Xu}), the module $M/MH$ has a projective cover $p\:Q\rarrow M/MH$ in $\Modr\R$.  Finally, $Q$ being projective, the map $p$ lifts to $f\:Q\rarrow M$ which is a projective cover of $M$ in $\Modr\R$ by~\cite[Corollary 15.13 and Lemma 27.5]{AF}.
\end{proof}

If projective covers of discrete simple modules exist over a given topological ring, finitely generated discrete modules over that ring behave very similarly to finitely generated modules over a classical semiperfect ring.
To this end, we recall that the \emph{radical} $\rad(M)$ of a module $M$ is the intersection of all maximal submodules of $M$.

\begin{prop} \label{structure-fg-discrete-prop}
Let\/ $\R$ be a topological ring such that each
simple discrete right $\R$\+module admits a projective cover
in\/ $\Modr\R$. Let further $H=H(\R)$ and $\HH=\HH(\R)$ be the abstract and the topological Jacobson radicals of\/ $\R$, respectively, and let $M$ be a finitely generated discrete right $\R$\+module.
Then $\rad(M)=MH=M\HH$ and $M/\rad(M)$ is semisimple.
\end{prop}

\begin{proof}
That $M/MH$ is semisimple was shown in the proof of Lemma~\ref{discrete-fg-projective-covers-provisional}. It follows that $MH$ is an intersection of certain maximal submodules of $M$, so $\rad(M)\subset MH$. As always $H\subset\HH$, we also have $MH\subset M\HH$. Finally, given any maximal submodule $N\subset M$, the simple factor $M/N$ is discrete and hence annihilated by $\HH$ thanks to~\cite[Lemma~7.2(ii)]{Pproperf}. It follows that $M\HH$ is contained in any maximal submodule of $M$, and so $M\HH\subset\rad(M)$.
\end{proof}

Now we prove in a few steps
(finishing in Proposition~\ref{discrete-fg-projective-covers})
that each topologically semiperfect
topological ring does admit projective covers of discrete simple
(and hence discrete finitely generated) modules.
We do not know whether the converse holds in general,
but we succeeded to prove it for topological rings with
a countable base of neighborhood of zero
(see Theorem~\ref{top-semiperfectnes-via-proj-covers-thm} below).

\begin{prop} \label{local-idempotent-times-Jacobson-radicals-agree}
 Let\/ $\R$ be a topologically semiperfect topological ring and
$e\in\R$ be a local idempotent.  Then
\begin{enumerate}
\renewcommand{\theenumi}{\alph{enumi}}
\item $eH(\R)=e\HH(\R)\subset e\R$;
\item $H(\R)e=\HH(\R)e\subset \R e$.
\end{enumerate}
\end{prop}

\begin{proof}
 Let $\HH=\HH(\R)$ be the topological Jacobson radical of $\R$ and
$\S=\R/\HH$ be the related topological quotient ring.
 By Theorem~\ref{topologically-semiperfect-structural-properties},
the ideal $\HH$ is strongly closed in $\R$ and the quotient ring
$\S$ is topologically semisimple.

 Let $g\in\S$ denote the image of $e\in\R$ under the natural
surjective homomorphism $\R\rarrow\S$.
 By Lemma~\ref{no-idempotents-in-topological-Jacobson}, we have
$e\notin\HH$,
%for all $z\in Z$;
so $g\ne0$.
 Hence the element~$g$ is a local idempotent in~$\S$.
 By Lemma~\ref{idempotents-in-topologically-semisimple-rings}(3--4),
the right $\S$\+module $g\S$ is discrete and simple, and
the left $\S$\+contramodule $\S g$ is simple.
 By Lemma~\ref{simple-contra-modules}, the left $\S$\+module
$\S g$ is simple, too.

 Thus the $\R$\+modules $g\S\simeq e\R/e\HH$ and $\S g\simeq\R e/\HH e$
are also simple.
 On the other hand, let $H=H(\R)$ be the abstract Jacobson radical.
 By Corollary~\ref{local-idempotent-cor}(4), the right $\R$\+module
$e\R/eH(\R)$ is simple (since $e\in\R$ is a local idempotent).
 Now we have a surjective map of simple right $\R$\+modules
$e\R/eH\rarrow e\R/e\HH$, which has to be an isomorphism;
hence $eH=e\HH$.
 Similarly, the left $\R$\+module $\R e/He$ is simple by
Corollary~\ref{local-idempotent-cor}(5), and it follows that
$He=\HH e$.
\end{proof}

\begin{cor} \label{Jacobson-closure}
 In any topologically semiperfect topological ring\/ $\R$,
the topological Jacobson radical\/ $\HH(\R)$ is equal to
the topological closure of the Jacobson radical $H(\R)$,
that is\/ $\HH(\R)=\overline{H(\R)}$.
\end{cor}

\begin{proof}
 Let $\mathbf e=(e_z)_{z\in Z}\in\R[[Z]]$ be a zero-convergent family
of orthogonal local idempotents such that $\sum_{z\in Z}e_z=1$
in~$\R$ (as in Theorem~\ref{topologically-semiperfect-rings}(3)).
 Then for any element $r\in\R$ we have $r=\sum_{z\in Z}re_z=
\sum_{z\in Z}e_zr$ (where both the infinite sums are understood as
the limits of finite partial sums in the topology of~$\R$).
 Given an element $h\in\HH(\R)$, we have $e_zh\in H(\R)$ and
$he_z\in H(\R)$ by
Proposition~\ref{local-idempotent-times-Jacobson-radicals-agree}.
\end{proof}

 An example showing that the Jacobson radical $H(\R)$ of
a topologically semiperfect topological ring $\R$ need not
be closed in $\R$, and therefore can be a proper subset of
the topological Jacobson radical $\HH(\R)$, will be given below in
Example~\ref{topological-jacobson-orthogonalization-counterex}(1).

At this point, we can also finish the proof of existence of projective covers of finitely generated discrete modules.

\begin{prop} \label{discrete-fg-projective-covers}
 Let\/ $\R$ be a topologically semiperfect topological ring.
 Then any finitely generated discrete right\/ $\R$\+module
has a projective cover in the category\/ $\Modr\R$.
\end{prop}

\begin{proof}
Lemma~\ref{discrete-fg-projective-covers-provisional} reduces that task
to proving the existence of projective covers of simple discrete
right $\R$\+modules.

 Let $\mathbf e=(e_z)_{z\in Z}\in\R[[Z]]$ be a zero-convergent family
of orthogonal local idempotents such that $\sum_{z\in Z}e_z=1$
in~$\R$.
 As in the proof of
Proposition~\ref{local-idempotent-times-Jacobson-radicals-agree},
we put $\HH=\HH(\R)$ and $\S=\R/\HH$.
 Let $g_z\in\S$ denote the image of $e_z\in\R$ under the natural
surjective homomorphism $\R\rarrow\S$.
 Then $\mathbf g=(g_z)_{z\in Z}\in\S[[Z]]$ is a zero-convergent
family of orthogonal primitive/local idempotents such that
$\sum_{z\in Z}g_z=1$ in~$\S$.

 By~\cite[Lemma~7.2(ii)]{Pproperf}, any simple discrete right
$\R$\+module is annihilated by $\HH$, so it comes from a simple
discrete right $\S$\+module.
 By Lemma~\ref{simples-over-top-semisimple-via-idempotents}(a) and
Proposition~\ref{local-idempotent-times-Jacobson-radicals-agree}(a),
it follows that any simple discrete right $\R$\+module has the form
$g_z\S\simeq e_z\R/e_z\HH=e_z\R/e_zH(\R)$ for some $z\in Z$.
 Applying Corollary~\ref{local-idempotent-cor} (property~(2) and
the uniqueness assertion at the end) we can conclude that any
simple right $\R$\+module of this form has a projective cover.
\end{proof}

Finally, we aim at proving that the existence of projective covers of discrete simple or finitely generated modules characterizes topological
semiperfectness when $\R$ has a countable base of neighborhoods of zero.
The argument is based on the following proposition.

\begin{prop} \label{countable-decomposition-of-idempotents-prop}
Let\/ $\R$ be a topological ring such that each
simple discrete right $\R$\+module admits a projective cover
in\/ $\Modr\R$.
 Given an idempotent element $e\in\R$ such that $e\R$ has a countable base of open neighborhoods of zero (as a subspace of\/~$\R$; this is always the case
if\/ $\R$ itself has a countable base of open neighborhoods of zero),
 then there exists a zero-convergent finite or countable family $e_1$,
$e_2$, $e_3$,~\dots\ of pairwise orthogonal local idempotents in\/
$\R$ such that $e = \sum_i e_i$.
\end{prop}

\begin{proof}
Let $\I\subset e\R$ be an open right submodule. Then $e\R/\I$ has a projective cover $P\rarrow e\R/\I$ by Lemma~\ref{discrete-fg-projective-covers-provisional}. In fact, the proof of the lemma reveals that $P$ is a finite direct sum $P=P_1\oplus\cdots\oplus P_n$ of projective covers of simple discrete right $\R$\+modules. Since the projection $\pi\:e\R\rarrow e\R/\I$ is a projective precover, we can without loss of generality assume that $P$ is a direct summand of $e\R$.
 More precisely, due to the projective precover properties of
the surjective morphisms of right $\R$\+modules $\bigoplus_{i=1}^nP_i
\rarrow e\R/\I$ and $\pi\:e\R\rarrow e\R/\I$, there exist
dotted arrows making the following two squares commutative
\[
 \xymatrix{
  \bigoplus\nolimits_{i=1}^n P_i \ar@{>..>}[r] \ar@{->>}[d] &
  e\R \ar@{..>>}[r] \ar@{->>}[d]^\pi &
  \bigoplus\nolimits_{i=1}^n P_i \ar@{->>}[d]
	\\
	e\R/\I \ar@{=}[r] &
	e\R/\I \ar@{=}[r] &
	e\R/\I
 } 
\]
Since $\bigoplus_{i=1}^n P_i\rarrow e\R/\I$ is
a cover, the composition $\phi\:\bigoplus_{i=1}^n P_i\rarrow e\R\rarrow
\bigoplus_{i=1}^n P_i$ is invertible.
 Replacing the morphism $e\R\rarrow\bigoplus_{i=1}^n P_i$ by its
composition with~$\phi^{-1}$ does not disturb commutativity of
the diagram, but makes the new composition
$\bigoplus_{i=1}^n P_i\rarrow e\R\rarrow\bigoplus_{i=1}^n P_i$
equal to the identity map.
 Now the projectors $e\R\rarrow P_i\rarrow e\R$ of the right
$\R$\+module $e\R$ onto its direct summands $P_i$ provide
local idempotents $e_i\in\R$, $1\le n$ such that
$e_1$,~\dots,~$e_n$ and $f = e-\sum_{i=1}^ne_i$ are pairwise orthogonal,
the restriction $\bigoplus_{i=1}^ne_i\R\rarrow e\R/\I$ of $\pi$ is a projective cover and the restriction $f\R\rarrow e\R/\I$ of $\pi$ vanishes. The latter condition means that $f\in\I$.

Next we choose a countable descending sequence $e\R\supset\I_0\supset\I_1\supset\I_2\supset\cdots$ of open right $\R$\+submodules 
which forms a base of open neighborhoods of zero in $e\R$.
In particular $\bigcap_{m=0}^\infty\I_m=0$ and, as in the previous paragraph, we find a pairwise orthogonal sequence of idempotents
$e_1$,~\dots, $e_{n_0}$, $f_0$ such that all $e_i$ are local, $e=\sum_{i=1}^{n_0} e_i+f_0$ and $f_0\in\I_0$. Since $f_0\R\cap\I_1$ is an open right submodule of $f_0\R$, we can likewise find pairwise orthogonal idempotents $e_{n_0+1}$,~\dots, $e_{n_1}$, $f_1$ such that all $e_i$ are local, $f_0=\sum_{i=n_0+1}^{n_1} e_i+f_1$ and $f_1\in\I_1$. Note that then also $e_i\in f_0\R\subset\I_0$ for all $n_0+1\le i\le n_1$ and
$e=\sum_{i=1}^{n_1} e_i+f_1$.
Iterating this procedure, if it turns out that $f_k\ne0$ for all
integers $k\ge0$, then we construct a countable sequence of pairwise 
orthogonal local idempotents $e_1$, $e_2$, $e_3$,~\dots\ which converges to zero in the topology of $\R$ and, by construction, the sum
$\sum_{i=1}^\infty e_i$ (which is by definition the limit of finite 
subsums) is equal to~$e$.
 Otherwise, if $f_k=0$, then $e=\sum_{i=1}^{n_k}e_i$, where
$e_1$,~\dots, $e_{n_k}$ is a finite sequence of pairwise orthogonal
local idempotents.
\end{proof}

\begin{thm}\label{top-semiperfectnes-via-proj-covers-thm}
Let $\R$ be a topological ring with a countable base of open neighborhoods
of zero. Then the following are equivalent:
\begin{enumerate}
\item $\R$ is topologically semiperfect;
\item each discrete simple right $\R$\+module has a projective cover in\/ $\Modr\R$;
\item each finitely generated discrete right $\R$\+module has a projective cover in the category\/ $\Modr\R$.
\end{enumerate}
\end{thm}

\begin{proof}
(1)\,$\Longrightarrow$\,(3)
Here we just repeat Proposition~\ref{discrete-fg-projective-covers}.

(2)\,$\Longrightarrow$\,(1)
By Proposition~\ref{countable-decomposition-of-idempotents-prop},
there exists a finite or countable zero-convergent family
$e_1$, $e_2$, $e_3$,~\dots\ of pairwise orthogonal local idempotents in\/
$\R$ such that $1 = \sum_i e_i$.
Hence, $\R$ is topologically semiperfect by the very definition in
Section~\ref{topologically-semiperfect-secn}.

(2)\,$\Longleftrightarrow$\,(3)
This is just Lemma~\ref{discrete-fg-projective-covers-provisional}.
\end{proof}

\Section{Projective Covers of Lattice-Finite Contramodules}

In this section we will give analogous results to those of Section~\ref{proj-cover-discrete}, but for contramodules. This partially restores the symmetry between the behavior of left and right modules for classical semiperfect rings. However, we need to be careful as to what is the correct contramodule analogue of a finitely generated discrete module. It turns out that the class of finitely generated contramodules is too wide and not so well behaved in some aspects. For the main results, we rather constrain ourselves to lattice-finite contramodules as defined in Subsection~\ref{finiteness-subsecn}. The reason for this is essentially captured the following lemma.

\begin{lem} \label{lattice-finite-proj-lem}
 Let\/ $\R$ be a topologically semiperfect topological ring and $e_1, \dots, e_n$ a~finite sequence of pairwise orthogonal local idempotents. Then the projective contramodule\/ $\coprod_{i=1}^n\R e_i$ is lattice-finite.
On the other hand, $\R$ itself is not a lattice-finite contramodule in general.
\end{lem}

\begin{proof}
Let $\P=\coprod_{i=1}^n\R e_i$ (we can consider the same finite direct sum in $\R\Modl$ and $\R\Contra$) and denote by $\HH=\HH(\R)$ the topological Jacobson radical.
It follows from the proof of Proposition~\ref{local-idempotent-times-Jacobson-radicals-agree}(b) that $\HH\P=H(\R)\P$ is a subcontramodule and the quotient is $\P/\HH\P=\bigoplus_{i=1}^n\S e_i$, where
$\S=\R/\HH$ is the related topologically semisimple topological quotient ring.
In particular, $\P/\HH\P$ is a finite sum of simple contramodules and, as any finite length contramodule, it is clearly lattice-finite.

Suppose now that $\P=\sum_{x\in X}\C_x$ for some set $X$ and a family of subcontramodules $\C_x\subset\P$. Then $\sum_{x\in X}(\C_x+\HH\P)/\HH\P = \P/\HH\P$, so there is a finite set $F\subset X$ such that $\P=\HH\P+\sum_{x\in F}\C_x=H(\R)\P+\sum_{x\in F}\C_x$. This is a sum of subcontramodules of $\P$, but as it is finite, it is also a sum of the underlying $\R$\+submodules (cf.\ Subsection~\ref{finiteness-subsecn}).
As $H(\R)\P$ coincides with the Jacobson radical of the underlying left $\R$\+module of $\P$ by~\cite[Proposition 17.10]{AF} and so is a superfluous $\R$\+submodule by~\cite[Theorem 10.4(1)]{AF}, we get $\P=\sum_{x\in F}\C_x$.
Hence, $\P$ is lattice-finite.

If $D$ is a division ring, $M=D^{(\omega)}\in D\Modl$ and $\R=\Hom_D(M,M)^\rop$ with the finite topology, then $\R$ is clearly topologically semiperfect (even topologically semisimple). If we denote by $M\larrow M\,:\!e_i$ the projectors to the copies of $D$ for $i<\omega$, then $\R=\sum_{i<\omega}\R e_i$ as a contramodule (since $1=\sum_{i<\omega}e_i$ in $\R$), but there is no finite subset $F\subset\omega$ for which $\sum_{i\in F}\R e_i$. So this particular topological ring $\R$ is not lattice-finite in $\R\Contra$
(and another topologically semiperfect topological ring of this kind is also considered in Example~\ref{topological-jacobson-orthogonalization-counterex}).
\end{proof}

Our next aim is to get a better control of simple contramodules over a topologically semiperfect topological ring (these are certainly lattice-finite). We start with general lemmas.

\begin{lem} \label{finitely-generated-contramodules-lemma}
 Let\/ $\R$ be a topological ring and\/ $\J\subset\R$ be a closed
right ideal.  Then
\begin{enumerate}
\renewcommand{\theenumi}{\alph{enumi}}
\item a left\/ $\R$\+contramodule is finitely generated if and only
if its underlying left\/ $\R$\+module is finitely generated;
\item for any finitely generated left\/ $\R$\+contramodule\/ $\C$,
one has\/ $\J\tim\C=\J\C$.
\end{enumerate}
\end{lem}

\begin{proof}
 Part~(a) holds because $\R[[X]]=\R[X]$ for a finite set~$X$.
 More precisely, one can say that a finite subset of
an $\R$\+contramodule generates it as a contramodule if and only
if it generates its underlying $\R$\+module.
 In part~(b), consider a surjective morphism of
left $\R$\+contramodules $f\:\P\rarrow\Q$.
 Then it is clear from the definitions that $\J\tim\Q=f(\J\tim\P)$
and $\J\Q=f(\J\P)$.
 In particular, let $X$ be a finite set for which there is
a surjective morphism of $\R$\+contramodules $f\:\R[[X]]\rarrow\C$.
 Then $\J\tim\C=f(\J\tim\R[[X]])=f(\J[[X]])=f(\J[X])=f(\J\R[X])=\J\C$.
\end{proof}

\begin{lem} \label{idempotents-decompose-contramodule-into-product}
 Let\/ $\R$ be a topological ring and\/ $\mathbf e=(e_z)_{z\in Z}
\in\R[[Z]]$ be a zero-convergent family of orthogonal idempotents
such that\/ $\sum_{z\in Z}e_z=1$ in\/~$\R$.
 Then, for any left\/ $\R$\+contramodule\/ $\C$, there is a natural
isomorphism of abelian groups\/ $\C\simeq\prod_{z\in Z}e_z\C$ given
by the map taking an element $c\in\C$ to the collection of elements
$(e_zc\in e_z\C)_{z\in Z}$.
 In particular, if\/ $\C\ne0$, then there exists $z\in Z$ for which
$e_z\C\ne0$.
\end{lem}

\begin{proof}
 This is explained~\cite[second paragraph of the proof of
Lemma~8.1(b)]{Pproperf}.
 Alternatively, the following construction allows to refer to
the assertion of~\cite[Lemma~8.1(b)]{Pproperf} rather than its proof.
 Consider the ring $\prod_{z\in Z}\boZ$, and endow it with the product
topology of the product of discrete rings of integers~$\boZ$.
 Then there exists a (unique) continuous homomorphism of topological
rings $\prod_{z\in Z}\boZ\rarrow\R$ given by the formula
$(n_z\in\boZ)_{z\in Z}\longmapsto\sum_{z\in Z}n_ze_z\in\R$.
 Hence the $\R$\+contramodule $\C$ becomes a contramodule over
$\prod_{z\in Z}\boZ$ via the contrarestriction of scalars
(see~\cite[Section~2.9]{Pproperf}).
 It remains to apply the description of contramodules over topological
products of topological rings given in~\cite[Lemma~8.1(b)]{Pproperf}
to the $\left(\prod_{z\in Z}\boZ\right)$\+con\-tra\-mod\-ule~$\C$.
\end{proof}

Now we are ready to observe an important relation between simple contramodules over a topologically semiperfect topological ring and simple contramodules over its topologically semisimple quotient modulo the topological Jacobson radical.

\begin{prop} \label{simple-contramodules-killed-by-top-Jacobson}
 Let\/ $\R$ be a topologically semiperfect topological ring,
$\HH=\HH(\R)$ be its topological Jacobson radical, and\/ $\S=\R/\HH$
be the related topological quotient ring.
 Then, for any simple left\/ $\R$\+contramodule\/ $\C$, one has\/
$\HH\tim\C=0$.
 In other words, any simple left\/ $\R$\+contramodule comes from
a (simple) left\/ $\S$\+contramodule via the contrarestriction of
scalars with respect to the natural surjective homomorphism of
topological rings\/ $\R\rarrow\S$.
\end{prop}

\begin{proof}
 Let $\mathbf e=(e_z)_{z\in Z}\in\R[[Z]]$ be a zero-convergent family
of orthogonal local idempotents such that $\sum_{z\in Z}e_z=1$
in~$\R$.
 By Lemma~\ref{idempotents-decompose-contramodule-into-product},
there exists $z\in Z$ such that $e_z\C\ne0$.
 Choosing a nonzero element in $e_z\C$, we can construct a surjective
morphism of $\R$\+contramodules $f\:\R e_z\rarrow\C$.
 By Lemma~\ref{simple-contra-modules}, \,$\C$ is a simple left
$\R$\+module; hence $f(H(\R)e_z)=H(\R)\C=0$.
 Applying Lemma~\ref{finitely-generated-contramodules-lemma}(b) and
Proposition~\ref{local-idempotent-times-Jacobson-radicals-agree}(b),
we conclude that $\HH\tim\C=\HH\C=f(\HH e_z)=f(H(\R)e_z)=0$.

 Alternatively, one can invoke
Proposition~\ref{local-idempotent-times-Jacobson-radicals-agree}(a)
to the effect that $e_zh\in H(\R)$ for all $h\in\HH$ and $z\in Z$.
 Hence $e_zhc\in H(\R)\C=0$ for all $c\in\C$, in view of
Lemma~\ref{simple-contra-modules}.
 According to
Lemma~\ref{idempotents-decompose-contramodule-into-product},
it follows that $hc=0$, and it remains to refer to
Lemma~\ref{finitely-generated-contramodules-lemma}(b).
\end{proof}

 We do \emph{not} know whether the assertion of
Proposition~\ref{simple-contramodules-killed-by-top-Jacobson}
holds true for an arbitrary (not necessarily topologically
semiperfect) topological ring~$\R$.
Now we, however, aim at proving the existence of projective covers
of simple contramodules.

\begin{lem} \label{fin-gen-modules-contramodules-proj-covers-the-same}
 Let\/ $\R$ be a topological ring and\/ $\C$ be a finitely generated
left\/ $\R$\+contramodule.
 Then\/ $\C$ has a projective cover in\/ $\R\Contra$ if and only if
the underlying left\/ $\R$\+module of\/ $\C$ has a projective
cover in\/ $\R\Modl$.
 The forgetful functor\/ $\R\Contra\rarrow\R\Modl$ takes any
projective cover of\/ $\C$ in\/ $\R\Contra$ to a projective cover
in\/ $\R\Modl$.
\end{lem}

\begin{proof}
 Firstly, one observes that any projective cover of a finitely
generated $\R$\+con\-tra\-mod\-ule $\C$ is a finitely generated
projective $\R$\+contramodule.
 Indeed, by the definition, $\C$ has a finitely generated projective
precover, and any projective cover of $\C$ is a direct summand of
any projective precover.
 For the same reason, any projective cover of a finitely generated
module is a finitely generated projective module.

 Secondly, for any finitely presented $\R$\+contramodule $\P$ and
any $\R$\+contramodule $\C$, the forgetful functor induces
an isomorphism on the Hom groups $\Hom^\R(\P,\C)\simeq\Hom_\R(\P,\C)$
\,\cite[Section~10]{PPT}.
 In particular, the forgetful functor restricts to an equivalence
between the categories of finitely generated projective left
$\R$\+contramodules and finitely generated projective left $\R$\+modules
(cf.\ the discussion in Section~\ref{finiteness-subsecn}).
 Therefore, any finitely generated projective precover of
the underlying $\R$\+module of $\C$ in $\R\Modl$ comes from a finitely
generated projective precover of $\C$ in $\R\Contra$.

 Finally, a projective precover $p\:\P\rarrow\C$ in $\R\Contra$ is
a projective cover if and only if, for any endomorphism
$f\:\P\rarrow\P$ in $\R\Contra$, the equation $pf=p$ implies that
$f$~is invertible (cf.~\cite[Lemma~4.1]{Pproperf}).
 Projective covers in $\R\Modl$ can be characterized similarly.
 Since $\Hom^\R(\P,\P)=\Hom_\R(\P,\P)$, it follows that $p$~is
a projective cover in $\R\Contra$ if and only if it is
a projective cover in $\R\Modl$.
\end{proof}

\begin{lem} \label{simple-contramodules-projective-covers}
 Let\/ $\R$ be a topologically semiperfect topological ring.
 Then any simple left\/ $\R$\+contramodule has a projective
cover in the abelian category\/ $\R\Contra$, as well as in
the abelian category\/ $\R\Modl$.
\end{lem}

\begin{proof}
 The proof is similar to that of
Proposition~\ref{discrete-fg-projective-covers}, except that
in addition one has to use
Proposition~\ref{simple-contramodules-killed-by-top-Jacobson}
and Lemma~\ref{fin-gen-modules-contramodules-proj-covers-the-same}.
 Let $\mathbf e=(e_z)_{z\in Z}\in\R[[Z]]$ be a zero-convergent family
of orthogonal local idempotents such that $\sum_{z\in Z}e_z=1$
in~$\R$.
 Put $\HH=\HH(\R)$ and $\S=\R/\HH$.
 Let $g_z\in\S$ denote the image of $e_z\in\R$ under the natural
surjective homomorphism $\R\rarrow\S$.

 By Proposition~\ref{simple-contramodules-killed-by-top-Jacobson},
any simple left $\R$\+contramodule comes from a simple
left $\S$\+contramodule.
 By Lemma~\ref{simples-over-top-semisimple-via-idempotents}(b) and
Proposition~\ref{local-idempotent-times-Jacobson-radicals-agree}(b),
it follows that any simple left $\R$\+contramodule has the form
$\S g_z\simeq \R e_z/\HH e_z=\R e_z/H(\R)e_z$ for some $z\in Z$.
 Applying Corollary~\ref{local-idempotent-cor} (property~(3) and
the uniqueness assertion at the end) we can conclude that any
simple left $\R$\+contramodule $\C$ of this form has a projective cover
\emph{as an\/ $\R$\+module}, i.~e., in the category $\R\Modl$.
 By Lemma~\ref{fin-gen-modules-contramodules-proj-covers-the-same},
it then follows that
% essentially
the same morphism is also
a projective cover of $\C$ in $\R\Contra$.
\end{proof}

Finally, our plan is to extend the existence of projective covers to all lattice-finite contramodules.

\begin{lem} \label{local-proj-contramodules-generate-lem}
 Let\/ $\R$ be a topologically semiperfect topological ring and
$X$ a set indexing all isomorphism classes
of simple left $\R$\+contramodules. For each $x\in X$, we denote
by $\P_x$ a projective cover of the corresponding simple contramodule.
Then $\{\P_x \mid x\in X\}$ is a set of generators for $\R\Contra$.
\end{lem}

\begin{proof}
 Let $\mathbf e=(e_z)_{z\in Z}\in\R[[Z]]$ be a zero-convergent family
of orthogonal local idempotents such that $\sum_{z\in Z}e_z=1$
in~$\R$.
 Then we know from Theorem~\ref{topologically-semiperfect-rings} that
$\R=\coprod_{z\in Z}\R e_z$.
 Moreover, using Corollary~\ref{local-idempotent-cor} and
Lemma~\ref{fin-gen-modules-contramodules-proj-covers-the-same} as in
the proof of the previous lemma, we see that each $\R e_z$ is
a projective cover of a simple contramodule.
 Finally, each $\R e_z$ is isomorphic to some $\P_x$ by the uniqueness
of projective covers.
\end{proof}

As a consequence, we have the following property of contramodules over topologically semiperfect topological rings which we do not expect to hold for a general topological ring.

\begin{lem}
 Let\/ $\R$ be a topologically semiperfect topological ring and $\C$ a nonzero contramodule. Then $\C$ has a simple subfactor.
\end{lem}

\begin{proof}
By the previous lemma, there is a nonzero homomorphism $f\:\P_x\rarrow\C$ for some $x\in X$. Since $\P_x$ has a unique maximal subcontramodule $H(\R)\P_x$ (see again the proof of Lemma~\ref{simple-contramodules-projective-covers}), so has it $f(\P_x)\simeq\P_x/\ker(f)$.
\end{proof}

Now we reach our main goal of the section and prove the existence of projective covers of lattice-finite contramodules.

\begin{prop} \label{lattice-finite-contramodules-projective-covers}
Let\/ $\R$ be a topologically semiperfect topological ring.
Then each lattice-finite left\/ $\R$\+contramodule has a projective cover in both the abelian categories\/ $\R\Contra$ and $\R\Modl$.
\end{prop}

\begin{proof}
 The proof is analogous to the one for
%Proposition~\ref{discrete-fg-projective-covers}.
Lemma~\ref{discrete-fg-projective-covers-provisional}.
Given $\C\in\R\Contra$, Lemma~\ref{local-proj-contramodules-generate-lem} yields a surjective homomorphism
$\coprod_{y\in Y}\P_y\rarrow\C$ in $\R\Contra$, where $Y$ is some indexing set and each $\P_y$ is a projective cover of a simple contramodule.
Suppose now that $\C$ is lattice-finite; then we find a surjective
homomorphism as above, but with $Y$ finite (as $\im(f)=\sum_{y\in Y}f(\P_y)$).
The induced homomorphism $\coprod_{y\in Y}\P_y/\HH\P_y\rarrow\C/\HH\C$, where $\HH\subset\R$ is the topological Jacobson radical, is also surjective.
Since $\coprod_{y\in Y}\P_y/\HH\P_y$ is a semisimple contramodule, so is $\C/\HH\C$.
Then, as finite direct sums of projective covers of modules are again projective covers, $\C/\HH\C$ has a projective cover $p\:\Q\rarrow\C/\HH\C$ as a left $\R$\+module by Lemma~\ref{simple-contramodules-projective-covers}, but then also as a contramodule
by Lemma~\ref{fin-gen-modules-contramodules-proj-covers-the-same}.
Finally, we lift $p$ to a morphism of contramodules $f\:\Q\rarrow\C$ using the projectivity of $\Q$ and prove exactly as in
%Proposition~\ref{discrete-fg-projective-covers}
Lemma~\ref{discrete-fg-projective-covers-provisional}
(using also the equality $\ker(p)=\HH\Q=H(\R)\Q$ given by Proposition~\ref{local-idempotent-times-Jacobson-radicals-agree}(b)) that $f$ is also a projective cover of left $\R$\+modules.
One further application of Lemma~\ref{fin-gen-modules-contramodules-proj-covers-the-same} tells us that $f$ is a projective cover of contramodules.
\end{proof}

We conclude the section by drawing consequences about the structure of lattice-finite contramodules. Given a contramodule $\C$, we denote by $\rad(\C)$ the intersection of all maximal subcontramodules of $\C$.

\begin{prop} \label{structure-lattice-finite-contra-prop}
Let\/ $\R$ be a topologically semiperfect topological ring and $\HH$ be the topological Jacobson radical.
\begin{enumerate}
\renewcommand{\theenumi}{\alph{enumi}}
\item A contramodule $\C\in\R\Contra$ is lattice-finite if and only if it is a quotient of a finite direct sum $\coprod_{i=1}^n \P_i$, where each $\P_i$ is a projective cover of a simple contramodule (i.e.\ $\P_i\simeq\R e_i$ for a local idempotent $e_i\in\R$).
\item If $\C$ is a lattice-finite left $\R$\+contramodule, then $\rad(\C)=\HH\C=\HH\tim\C$ and $\C/\rad(\C)$ is a semisimple contramodule.
\end{enumerate}
\end{prop}

\begin{proof}
(a) Each contramodule of the form $\coprod_{i=1}^n \P_i$ is lattice-finite by Lemma~\ref{lattice-finite-proj-lem} and so is every quotient.
If, on the other hand, $\C$ is lattice-finite, it has a projective cover of the form $\coprod_{i=1}^n \P_i\rarrow\C$ by the proof of Proposition~\ref{lattice-finite-contramodules-projective-covers}.

(b) As any lattice-finite contramodule $\C$ is finitely generated, we have $\HH\C=\HH\tim\C$ by Lemma~\ref{finitely-generated-contramodules-lemma}(b).
The quotient $\C/\HH\C$ is semisimple by the proof of Proposition~\ref{lattice-finite-contramodules-projective-covers}, so $\rad(\C)\subset\HH\C$.
On the other hand, if $\D\subset\C$ is any maximal subcontramodule, then $\HH\tim(\C/\D)=0$ by Proposition~\ref{simple-contramodules-killed-by-top-Jacobson}, so $\HH\tim\C\subset\D$. In particular, $\HH\tim\C\subset\rad(\C)$.
\end{proof}

\Section{Lifting Idempotents modulo the Topological Jacobson Radical}
\label{lifting-idempotents-secn}

In this section, we will discuss results on lifting idempotent elements 
modulo the topological Jacobson radical. We start with an easy consequence
of Proposition~\ref{discrete-fg-projective-covers}.

\begin{prop} \label{lifting-orthogonal-primitive-idempotents}
 Let\/ $\R$ be a topologically semiperfect topological ring,
$\HH=\HH(\R)$ be its topological Jacobson radical, and\/ $\S=\R/\HH$
be the related topological quotient ring.
 Then any finite orthogonal family of primitive idempotents in\/ $\S$
can be lifted modulo\/ $\HH$ to a finite orthogonal family of
local idempotents in\/~$\R$.
\end{prop}

\begin{proof}
 Recall that, by
Theorem~\ref{topologically-semiperfect-structural-properties},
the topological ring $\S$ is topologically semisimple.
 Let $(g'_w)_{w\in W}$ be a finite orthogonal family of primitive
idempotents in~$\S$.
 Then, by Lemma~\ref{idempotents-in-topologically-semisimple-rings}(3)
the right $\S$\+module
$g'_w\S$ is discrete and simple for every $w\in W$.
 By Proposition~\ref{discrete-fg-projective-covers},
the discrete and simple right $\R$\+module $g'_w\S$ has a projective cover
$p_w\:P_w\rarrow g'_w\S$ in $\Modr\R$.

 By~\cite[Remark~1.4.2]{Xu}, a finite direct sum of covers is a cover;
so the morphism
$$
 \bigoplus\nolimits_{w\in W}p_w\:\bigoplus\nolimits_{w\in W}P_w
 \lrarrow\bigoplus\nolimits_{w\in W}g'_w\S
$$
is a projective cover in $\Modr\R$.
 Now we have a natural split epimorphism of right $\S$\+modules
$\S\rarrow\bigoplus_{w\in W}g'_w\S$ given by the formula
$s\longmapsto (g'_ws)_{w\in W}$ for all $s\in\S$.
 The composition $\R\rarrow\S\rarrow\bigoplus_{w\in W}g'_w\S$ is
a projective precover in $\Modr\R$.
 Hence the direct sum of $\R$\+modules $\bigoplus_{w\in W}P_w$
is a direct summand of the free right $\R$\+module~$\R$.

 More precisely, due to the projective precover properties of
the surjective morphisms of right $\R$\+modules $\bigoplus_{w\in W}P_w
\rarrow\bigoplus_{w\in W}g'_w\S$ and $\R\rarrow\S$, there exist
dotted arrows making the following two squares commutative
\begin{equation} \label{eq-lifting-idempotents}
\begin{gathered}
 \xymatrix{
  \bigoplus\nolimits_{w\in W}P_w \ar@{->>}[r]
  & \bigoplus\nolimits_{w\in W}g'_w\S \\
  \R \ar@{->>}[r] \ar@{..>>}[u] & \S \ar@{->>}[u]
 } 
 \qquad
 \xymatrix{
    \bigoplus\nolimits_{w\in W}P_w \ar@{->>}[r] \ar@{>..>}[d]
    & \bigoplus\nolimits_{w\in W}g'_w\S \ar@{>->}[d] \\
  \R \ar@{->>}[r] & \S
 }
\end{gathered}
\end{equation}
where the split monomorphism of right $\S$\+modules
$\bigoplus_{w\in W}g'_w\S\rarrow\S$ is given by the obvious
rule $(s_w)_{w\in W}\longmapsto\sum_{w\in W}s_w$ for all
$s_w\in g'_w\S$.
 The composition $\bigoplus_{w\in W}g'_w\S\rarrow\S\rarrow
\bigoplus_{w\in W}g'_w\S$ is the identity map.

 Since $\bigoplus_{w\in W}P_w\rarrow\bigoplus_{w\in W}g'_w\S$ is
a cover, the composition $\phi\:\bigoplus_{w\in W}P_w\rarrow\R\rarrow
\bigoplus_{w\in W}P_w$ is invertible.
 Replacing the morphism $\R\rarrow\bigoplus_{w\in W}P_w$ by its
composition with~$\phi^{-1}$ does not disturb commutativity of
the leftmost square diagram, but makes the new composition
$\bigoplus_{w\in W}P_w\rarrow\R\rarrow\bigoplus_{w\in W}P_w$
equal to the identity map.
 We will use this specific way to view the right $\R$\+module
$\bigoplus_{w\in W}P_w$ as a direct summand of the right
$\R$\+module~$\R$.

 Now the projectors $\R\rarrow P_w\rarrow\R$ of the right
$\R$\+module $\R$ onto its direct summands $P_w$ provide
the desired finite family of orthogonal idempotents
$(e'_w\in\R)_{w\in W}$ lifting the idempotents $g'_w\in\S$.
 Finally, the idempotents $e'_w\in\R$ are local by
Corollary~\ref{local-idempotent-cor}(2), since the right $\R$\+module
$P_w=e'_w\R$ is a projective cover of the simple right
$\R$\+module~$g'_w\S$.
\end{proof}

With a little more work, we can refine the previous lifting result
as follows.

\begin{lem} \label{strong-lifting-of-idempotents}
 Let\/ $\R$ be a topologically semiperfect topological ring,
$\HH=\HH(\R)$ be its topological Jacobson radical, and\/ $\S=\R/\HH$
be the related topological quotient ring.
Given finitely many elements $(f'_w)_{w\in W}$ in\/ $\R$ such that
$g'_w:=f'_w+\HH$ form an orthogonal family of primitive idempotents
in\/ $\S$, then there exists an orthogonal family of local idempotents
$(e'_w)_{w\in W}$ in\/ $\R$ such that $e'_w\in f'_w\R$ and
$e'_w+\HH=g'_w$ for each $w\in W$. 
\end{lem}

\begin{proof}
The trick is that we can choose the dotted morphism in the rightmost
square of~\eqref{eq-lifting-idempotents} in such a way that
$\im(P_w\to\R)\subset f'_w\R$. Indeed, the image of the composition
$P_w\rarrow g'_w\S\rarrow \S$ equals $g'_w\S$, so it factors through
the epimorphism $f'_w\R\rarrow g'_w\S$.

From this point on, we can continue as in the proof of
Proposition~\ref{lifting-orthogonal-primitive-idempotents} and
the constructed projectors $\R\rarrow P_w\rarrow\R$ will have image
contained in $f'_w\R$. Hence, the lifted idempotents will satisfy
$e'_w\in f'_w\R$.
\end{proof}

Using the latter lemma, we can lift convergent infinite families of
primitive idempotents. Note, however, that the orthogonality is not
under control here.

\begin{thm}\label{lifting-convergent-families-of-idempotents}
Let\/ $\R$ be a topologically semiperfect topological ring,
$\HH=\HH(\R)$ be its topological Jacobson radical, $\S=\R/\HH$
the related topological quotient ring,
and let\/ $\mathbf g'=(g'_z)_{z\in Z}\in\S[[Z]]$ be a zero-convergent
family of primitive idempotents.
Then there exists a zero-convergent family\/
$\mathbf e'=(e'_z)_{z\in Z}\in\R[[Z]]$
of local idempotents which lifts\/~$\mathbf g'$ modulo\/~$\HH$.
\end{thm}

\begin{proof}
Since $\HH$ is strongly closed in $\R$ by Theorem~\ref{topologically-semiperfect-structural-properties}, we can lift $\mathbf g'$ to a zero-convergent family of elements $\mathbf f'=(f'_z)_{z\in Z}\in\R[[Z]]$.
By Lemma~\ref{strong-lifting-of-idempotents}, we can lift each primitive idempotent $g'_z=f'_z+\HH$ individually to a local idempotent
$e'_z\in f'_z\R$. Since $\R$ has a base of open neighborhoods of zero
formed by right ideals, the family $\mathbf e'=(e'_z)_{z\in Z}$ of local idempotents in $\R$ is still zero-convergent.
\end{proof}

\begin{rem}
 Let\/ $\R$ be a topologically semiperfect topological ring with
the topological Jacobson radical $\HH$ and the related quotient ring
$\S=\R/\HH$.
 It would be interesting to know whether an arbitrary zero-convergent 
orthogonal family of (primitive) idempotents in $\S$ can be lifted to
a zero-convergent orthogonal family of (local) idempotents in~$\R$.

 In particular, let us say that a zero-convergent orthogonal family is
\emph{complete} if their sum is equal to~$1$.
 Can one lift any complete zero-convergent family of primitive
idempotents in $\S$ to a complete zero-convergent family of local
idempotents in~$\R$\,?

 The problem of lifting orthogonal idempotents is discussed in
the papers~\cite{MM2,CN}, and the conclusion seems to be that it is
easier to orthogonalize lifted idempotents than to lift individual
idempotents.
 However, this heuristic may be only applicable to lifting orthogonal
idempotents \emph{modulo an ideal contained in the Jacobson radical}.

 Dealing with topologically semiperfect topological rings, one encounters
the problem that the topological Jacobson radical $\HH$ can be strictly
larger than the abstract Jacobson radical~$H$.
 The following example illustrates some of the difficulties.
\end{rem}

\begin{ex} \label{topological-jacobson-orthogonalization-counterex}
 Let $A$ be a commutative Noetherian discrete valuation domain (e.~g.,
$A=k[t]_{(t)}$ is the localization of the polynomial ring $k[t]$ over
a field~$k$ with respect to its maximal ideal $(t)\subset k[t]$).
 Let $t\in A$ be a uniformizing element.

 Consider the free $A$\+module $M=A^{(\omega)}=\bigoplus_{i=0}^\infty A$
with a countable basis $(b_i)_{i=0}^\infty$.
 Obviously, $M$ is a direct sum of $A$\+modules with local endomorphism
rings.
 By Proposition~\ref{semiperfect-decomposition}, the endomorphism ring
$\R=\End_A(M)^\rop$, endowed with the finite topology, is a topologically
semiperfect topological ring.
 The elements of the ring $\R$ are the row-finite $\omega\times\omega$
matrices with the entries in the ring~$A$.

 It is clear from the description of the topological Jacobson radical
obtained in the proof of
Theorem~\ref{topologically-semiperfect-structural-properties} that
the topological Jacobson radical $\HH=\HH(\R)$ consists of all
the matrices with the entries divisible by~$t$, or in other words, of
all the elements divisible by $t$ in $\R$, that is, $\HH=t\R$.

\smallskip
 (1)~Here is an example of an element $h\in\HH(\R)$ which does not belong
to the Jacobson radical $H(\R)$ of the ring $\R$ viewed as an abstract
ring.
 Consider the linear map $M\larrow M\,:\!h$ given by the formula
$tb_{i+1}\mapsfrom b_i$ for all $i\in\omega$.
 Then $h\in\HH$, but the map $M\larrow M\,:\!1-h$ is not invertible.
 In fact, the map $1-h$ is a locally split monomorphism of $A$\+modules
(as it should be by
Lemma~\ref{topological-Jacobson-via-locally-split-monos},
cf.~\cite[Proposition~4.1 and Lemma~4.4]{BPS}), but it is not surjective:
the cokernel of $1-h$ is isomorphic to the field of fractions
$A[t^{-1}]$ of the ring~$A$.

 Notice that, in a topological ring with a two-sided linear topology
(i.~e., a base of neighborhoods of zero consisting of two-sided ideals),
the topological Jacobson radical $\HH$ always coincides with the abstract
Jacobson radical $H$ by~\cite[Theorem~3.8(3)]{IMR} (cf.\ the discussions
in~\cite[Section~1]{Gre} and~\cite[Section~7]{Pproperf}).
 The above counterexample shows that this is not true for right linear
topological rings in general.

\smallskip
 (2)~Here is an example showing that the technology
of~\cite[Proposition~8.2]{PS3} does not resolve the problem of lifting
infinite families of orthogonal idempotents modulo~$\HH$.

 According to the discussion in the proofs of
Proposition~\ref{matrices-of-noninvertibles}
and Theorem~\ref{topologically-semiperfect-structural-properties},
the topological quotient ring $\S=\R/\HH$ is naturally isomorphic
to the endomorphism ring $\End_k(M/tM)^\rop$ of
the infinite-dimensional vector space $M/tM$ over the residue field
$k=A/tA$, with the finite topology on the endomorphism ring.
 Put $\bar b_i=b_i+tM\in M/tM$, so the elements~$\bar b_i$,
\,$i\in\omega$, form a basis of $M/tM$.
 
 Consider the complete orthogonal family of primitive idempotents
$M/tM\larrow M/tM\,:\!g_j$ in $\S$, \ $j\in\omega$, defined by
the obvious rules $\bar b_j\mapsfrom\bar b_j\,:\!g_j$ and
$0\mapsfrom\bar b_i\,:\!g_j$ for $i\ne j$.
 There is a trivial lifting of this family of idempotents to
a complete orthogonal family of local idempotents $M\larrow M\,:\!e_j$
in $\R$ defined by the similar formulas  $b_j\mapsfrom b_j\,:\!e_j$
and $0\mapsfrom b_i\,:\!e_j$ for $i\ne j$.
 So $\mathbf e=(e_j\in\R)_{j\in\omega}\in\R[[\omega]]$ is a family
of idempotents in $\R$ satisfying the conditions of
Theorem~\ref{topologically-semiperfect-rings}(3).
 But how does one arrive to such a ``good'' lifting of a family of
idempotents in the general setting, and what if one accidentally starts
with choosing a bad lifting instead?

 Here is an example of a ``bad'' lifting: let the idempotent
endomorphisms $M\larrow M\,:\!e'_j$ be given by the formulas
$b_j-tb_{j+1}\mapsfrom b_j\,:\!e'_j$ and 
$0\mapsfrom b_i\,:\!e'_j$ for $i\ne j$.
 Then $\mathbf e'=(e'_j\in\R)_{j\in\omega}\in\R[[\omega]]$ is
a zero-convergent, but \emph{nonorthogonal} family of local idempotents
in~$\R$ with $e'_j+\HH=g_j$ for all $j\in\omega$.
 One cannot orthogonalize the family~$\mathbf e'$
using~\cite[Proposition~8.2]{PS3}, because
$u=\sum_{j=0}^\infty e'_z=1-h\in1-\HH$ is not an invertible element
in~$\R$ (see the discussion in~(1)).
\end{ex}

\Section{Structure of Projective Contramodules}
\label{structure-of-projective-contramodules-secn}

We conclude the paper by the following structure theorem for projective contramodules over topologically semiperfect topological rings with a countable base of open neighborhoods of zero,
which generalizes~\cite[Theorem 27.11]{AF}.

\begin{thm} \label{projective-contramodules-for-countable-basis-thm}
Let $\R$ be a topologically semiperfect topological ring with a countable base of open neighborhoods of zero and
let $\P$ be a projective contramodule.
Then $\P$ decomposes as $\P=\coprod_{x\in X}\P_x$, where each
$\P_x$ is a projective cover of a simple contramodule
(i.e.\ $\P_x\simeq\R e_x$ for a local idempotent $e_x\in\R$).
\end{thm}

\begin{rems}
 (1)~By~\cite[Corollary 4.4 and Remark 4.5]{PS3}, we can realize $\R$ as the topological endomorphism ring $\R=\Hom_A(M,M)^\rop$ (equipped with the finite topology) of a countably generated left module $M$ over a ring $A$.
Since $\R$ is topologically semiperfect, $M$ is a direct sum of modules with local endomorphism rings (Proposition~\ref{semiperfect-decomposition}).
Furthermore, \cite[Theorem~3.14(iii)]{PS3} says that there is an equivalence $\Add(M)\simeq\R\Contra_\proj$, where $\Add(M)\subset A\Modl$ is the full subcategory given by all direct summands of direct sums of copies of $M$.
Keeping this equivalence in mind, Theorem~\ref{projective-contramodules-for-countable-basis-thm} can be in fact quickly deduced from the Crawley--J\o{}nsson--Warfield \cite[Theorem 26.5]{AF},
which says that every $N\in\Add(M)$ is a direct sum of modules with local endomorphism rings.
However, Proposition~\ref{countable-decomposition-of-idempotents-prop} 
allows us to provide a direct contramodule-based argument for a key step.

\smallskip
 (2)~We do not know whether the same result holds without the assumption of $\R$ having a countable base of open neighborhoods of zero.
In that case we can still assume by~\cite[Corollary 4.4]{PS3} that $\R=\Hom_A(M,M)^\rop$ for a left module $M$ over a ring $A$ such that $M$ is a direct sum of modules with local endomorphism rings,
but now $M$ need not be countably generated.
There is still an equivalence $\Add(M)\simeq\R\Contra_\proj$, so the question again translates to the question whether every module in $N\in\Add(M)$ is a direct sum of modules with local endomorphism rings.
This appears to be a long standing open problem (see for instance the first paragraph of~\cite[Section 26]{AF}).

\smallskip
 (3)~Suppose that $\Q$ is a projective left $\R$\+contramodule over a topologically semiperfect topological ring and
denote by $\HH=\HH(\R)$ the topological Jacobson radical and by $\S=\R/\HH$ the corresponding topologically semisimple quotient ring.
Then $\Q/\HH\tim\Q$ comes from an $\S$\+contramodule by a contrarestriction of scalars with respect to the surjection $\R\rarrow\S$, so it is a semisimple $\R$\+contramodule.
That is, $\Q/\HH\tim\Q=\coprod_{z\in Z}\C_z$ with each $\C_z$ simple in $\R\Contra$.
As we know from Lemma~\ref{simple-contramodules-projective-covers}, for each $z\in Z$ we have a projective cover $\P_z\rarrow\C_z$, so there is also a surjective homomorphism $\P:=\coprod_{z\in Z}\P_z\rarrow\coprod_{z\in Z}\C_z$ with $\P$ a projective contramodule, and it is not difficult to check that the kernel is $\HH\tim\P$.

Thus, we have two projective left $\R$\+contramodules $\P,\Q$ with
$\P/\HH\tim\P=\coprod_{z\in Z}\C_z=\Q/\HH\tim\Q$ and the question in~(2) can be reformulated to the problem of whether $\P\simeq\Q$ in $\R\Contra$.
If the natural surjections $\P\rarrow\P/\HH\tim\P$ and $\Q\rarrow\Q/\HH\tim\Q$ were projective covers, $\P$ and $\Q$ would have to be isomorphic.
This happens for instance if the set $Z$ above is finite, but it fails in general.
For instance, the surjection $\R\rarrow\R/\HH$ is a projective cover if and only if $\HH$ coincides with the abstract Jacobson radical $H(\R)$.
That can be seen either from Lemma~\ref{fin-gen-modules-contramodules-proj-covers-the-same} and the well known fact that $H(\R)$ is the maximal superfluous left $\R$\+submodule of $\R$,
or by~\cite[Lemma 2.5]{BPS}.
However, we presented in Example~\ref{topological-jacobson-orthogonalization-counterex} an example of a topologically semiperfect topological ring, even one with a countable base of open neighborhoods of zero, for which $H(\R)\ne\HH$.
\end{rems}

Our main point here is that as a consequence
of Proposition~\ref{countable-decomposition-of-idempotents-prop},
a special case of Theorem~\ref{projective-contramodules-for-countable-basis-thm} for projective contramodules which have a countable generating set follows quickly.

\begin{lem} \label{countable-decomposition-for-countable-basis-lem}
Let $\R$ be a topologically semiperfect topological ring with a countable base of open neighborhoods of zero and
let $\P$ be a countably generated projective contramodule
(that is $\P$ is a direct summand of $\R[[\omega]]$ in $\R\Contra$).
Then there is a decomposition $\P=\coprod_{x\in X}\P_x$, where each
$\P_x$ is a projective cover of a simple contramodule.
\end{lem}

\begin{proof}
Here we use Morita theory for contramodules from~\cite[Section 5]{PS3} (see also \cite[Section 7.3]{PS1}).
Since $\R\Contra_\proj$ is a topologically agreeable additive category by \cite[Remark 3.12]{PS3}, $\S=\Hom^\R(\R[[\omega]],\R[[\omega]])^\rop$ is naturally a complete separated right linear topological ring.
In fact, $\S$ can be identified with the ring of row-zero-convergent matrices of size $\omega\times\omega$ and the topology was explicitly described in \cite[Lemma 5.1]{PS3}.
As a consequence, we observe that $\S$ also has a countable base of neighborhoods of zero.
Furthermore, by \cite[Proposition~5.3 and Lemma~5.4]{PS3}, there is an equivalence of categories
$\R\Contra\simeq\S\Contra$ which sends $\R[[\omega]]$ to $\S$.
In particular, $\P$ is sent to a direct summand of $\S$ in $\S\Contra$, so to a projective $\S$-contramodule of the form $\S e$ for an idempotent $e=e^2\in\S$.

Now we apply Proposition~\ref{countable-decomposition-of-idempotents-prop} and express $e=\sum_ie_i$, where $e_1$, $e_2$, $e_3$,~\dots\ is a zero-convergent finite or countable family of pairwise orthogonal local idempotents in $\S$.
Using an obvious modification of the argument for Theorem~\ref{topologically-semiperfect-rings}(3)\,$\Longrightarrow$\,(1),
we obtain a decomposition $\S e=\coprod_{i=1}^N\S e_i$ in $\S\Contra$, where each $\S e_i$ is a projective cover of a simple $\S$\+contramodule.
Applying the equivalence $\R\Contra\simeq\S\Contra$ once again,
we obtain the desired decomposition of~$\P$ in $\R\Contra$.
\end{proof}

In order to prove Theorem~\ref{projective-contramodules-for-countable-basis-thm} in full generality, we will need a version of a structure theorem of Kaplansky~\cite[Corollary 26.2]{AF} for projective contramodules.
Although one can certainly give a direct contramodule-based proof, we prefer to reduce it a statement about modules.
To this end, if $M$ is a left $A$\+module and $\R=\Hom_A(M,M)^\rop$ equipped with the finite topology, we denote by $\Psi\:\Add(M)\rarrow\R\Contra_\proj$ the equivalence from \cite[Theorem~3.14(iii)]{PS3})
(it was mentioned above several times).

\begin{lem} \label{kappa-generated-projective-contramodules}
Let $\kappa$ be an infinite cardinal, $A$ a ring,
$M$ a left $A$\+module generated by a set of cardinality at most $\kappa$, and let $\R=\Hom_A(M,M)^\rop$ be the topological endomorphism ring of $M$ equipped with the finite topology.
Then the following are equivalent for an $A$\+module $N\in\Add(M)$:
\begin{enumerate}
\item $N$ is generated by a set of cardinality at most $\kappa$ as a left $A$\+module,
\item $\Psi(N)$ is generated by a set of cardinality at most $\kappa$ as a left $\R$\+contramodule.
\end{enumerate}
\end{lem}

\begin{proof}
Note that $N$ satisfies (1) if and only if $N$ is a direct summand of $M^{(\kappa)}$ in $A\Modl$.
Since $\Psi(N)$ is a projective contramodule, it similarly satisfies (2) if and only if it is a direct summand of $\R[[\kappa]]$.
The conclusion then follows from the fact that the equivalence $\Psi\:\Add(M)\rarrow\R\Contra_\proj$ sends $M^{(\kappa)}$ to $\R[[\kappa]]$.
\end{proof}

\begin{prop}[Kaplansky] \label{kaplansky-for-projective-contramodules-prop}
Let $\kappa$ be an infinite cardinal and $\R$ be a topological ring with a base of neighborhoods of zero of cardinality at most $\kappa$.
Then each projective contramodule $\P$ is a coproduct in $\R\Contra$ of a family of (projective) contramodules generated by sets of cardinalities at most $\kappa$.
\end{prop}

\begin{proof}
By~\cite[Corollary 4.4]{PS3}, there is a ring $A$ and a left $A$\+module $M$ such that there is an isomorphism of topological rings $\R\simeq\Hom_A(M,M)^\rop$, where the latter is equipped with the finite topology.
Moreover, thanks to~\cite[Remark 4.5]{PS3} we can choose $A$ and $M$ so that $M$ has a generating set of cardinality at most~$\kappa$ as a left $A$\+module.
Now, given any $\P\in\R\Contra_\proj$, we employ the equivalence $\Psi\:\Add(M)\rarrow\R\Contra_\proj$ and find $N\in\Add(M)$ such that $\Psi(N)\simeq\P$.
By a theorem of Kaplansky for modules \cite[Theorem 26.1]{AF},
there exists a direct sum decomposition $N=\bigoplus_{x\in X}N_x$, where each $N_x$ possesses a generating set of cardinality at most $\kappa$.
If we denote $\P_x=\Psi(N_x)$ for each $x\in X$, the equivalence $\Psi$ transfers this decomposition to an isomorphism of contramodules
$\P\simeq\coprod_{x\in X}\P_x$ and each $\P_x$ is generated by a set of cardinality at most $\kappa$ by Lemma~\ref{kappa-generated-projective-contramodules}.
\end{proof}

Finally, we can complete the proof of the theorem.

\begin{proof}[Proof of Theorem~\ref{projective-contramodules-for-countable-basis-thm}]
Let $\R$ be a topologically semiperfect topological ring with a countable base of open neighborhoods of zero and let $\P$ be a projective contramodule.
By Proposition~\ref{kaplansky-for-projective-contramodules-prop} applied to $\P$ with $\kappa=\omega$, we obtain a decomposition $\P=\coprod_{z\in Z}\P_z$, where each $\P_z$ is a countably generated projective contramodule.
It remains to apply Lemma~\ref{countable-decomposition-for-countable-basis-lem} which says that each $\P_z$ in turn decomposes as $\P_z=\coprod_{x\in X_z}\P_{z,x}$, where each $\P_{z,x}$ is a projective cover of a simple contramodule.
\end{proof}

\end{document}